\documentclass[twoside,12pt]{amsart}

\date{\today}

\usepackage[latin1]{inputenc}
\usepackage{pstricks}
\usepackage{enumerate}	
\usepackage{graphicx,color}
\usepackage[english]{babel}
\usepackage[toc,page]{appendix}
\usepackage{amssymb, mathrsfs, amsfonts, amsmath,amsthm}
\usepackage{amsbsy, hyperref}
\usepackage{latexsym}
\usepackage{booktabs}
\usepackage{tikz}

\usepackage{eurosym} 

\usepackage{newfloat}

\newtheorem{Theorem}{Theorem}[section]

\newtheorem{Definition}{Definition}[section]

\newtheorem{Remark}{Remark}[section]
\newtheorem{Lemma}{Lemma}[section]
\newtheorem{claim}{Claim}[Lemma]

\newcommand{\B}{\mathbb{B}}
\newcommand{\R}{\mathbb{R}}

\newcommand{\n}{\mathbb{N}}

\newcommand{\ee}{\mathbf e}

\newcommand{\dist}{\operatorname{dist}}
\newcommand{\spt}{\operatorname{spt}}

\author[E. S. Gama]{Eddygledson S. Gama}
\address[Gama]{
  Departamento de Matem\'atica,
  Universidade Federal do Cear\'a, Bloco 914, Campus do Pici,
  Fortaleza, Cear\'a, 60455-760, Brazil.
}
\email{eddygledson@gmail.com}
\author[F. Mart\'in]{Francisco Mart\'\i{}n}
\address[Mart\'in]{
  Departamento de Geometr\'\i{}a y Topolog\'\i{}a,
  Universidad de Granada,
  18071 Granada, Spain.
}
\email{fmartin@ugr.es}
\thanks{
E. S. Gama is supported by Capes/PDSE/88881.132464/2016-01. F. Mart\'in is partially supported by the
  MINECO/FEDER grant MTM2014-52368-P and  by the
Leverhulme Trust grant IN-2016-019. }

\title[Translating solitons asymptotic to hyperplanes in $\R^{n+1}$]{Translating solitons of the mean curvature flow asymptotic to hyperplanes in $\R^{n+1}$}

\begin{document}
\maketitle

\begin{abstract} A translating soliton is a hypersurface $M$ in $\R^{n+1}$ such that 
the family $M_t= M- t \,\ee_{n+1}$
is a mean curvature flow, i.e., such that normal component of the velocity at each point is
equal to the mean curvature at that point $\mathbf{H}=\ee_{n+1}^{\perp}.$
In this paper we obtain a characterization of hyperplanes which are parallel to the velocity and the family of tilted grim reaper cylinders as the only translating solitons in $\mathbb{R}^{n+1}$ which are $C^1$-asymptotic to two half-hyperplanes outside a non-vertical cylinder. This result was proven for translators in $\mathbb{R}^3$  by the second author, Perez-Garcia, Savas-Halilaj and Smoczyk  under the additional hypotheses that the genus of the surface was locally bounded and the cylinder was perpendicular to the translating velocity. \end{abstract}

\section{Introduction}

Consider $F(\cdot,t): M^n \rightarrow \R^{n+1}$ a one-parameter family of smooth immersed hypersurfaces in $\R^{n+1}$. We say that $M_t:= F(M,t)$ is a mean curvature flow if it satisfies:
\begin{eqnarray*}
\partial_t F(p,t) &= &\mathbf{H}(p,t), \; \; p \in M, \; t >0, \\
F(p,0) &= & F_0(p),
\end{eqnarray*}
where $\mathbf{H}(p,t)$ means the mean curvature vector of the hypersurface $M_t$ at the point $F(p,t)$, and $F_0$ is a given initial immersion. It is well known (see \cite{WHINotes} for instance) that if the initial hypersurface is compact, then the flow develops singularities in finite time. In particular we have that
$|A(p,t)|$ is not bounded when we approach the maximal time $T$. Singularities are classified according to the rate
at which $\max_{p \in M_t}|A(p,t)|$ blows up. If there is a constant $C>1$ such that 
$$\max_{p\in M_t}|A(p,t)| \sqrt{2(T-t)} \leq C,$$
then we say that the flow develops a  \textit{Type I singularity} at instant $T$.
Otherwise, that is, if
$$\limsup_{t\to T} \max_{p\in M_t}|A(p,t)|\sqrt{(T-t)}=+\infty,$$
we say that is a \textit{Type II singularity}.

An standard example a Type II singularity is given by a loop pinching off to a cusp (see Figure \ref{cardiod}). S. Angenent \cite{An1} proved, in the case of convex planar curves, that singularities of this kind are asymptotic (after rescaling) to the grim reaper curve $y=-log (\cos x)$, $x \in (-\pi/2,\pi/2),$ which moves set-wise by translation. In this case, up to inner diffeomorphisms of the soliton, it can be seen as an eternal solution of the curve shortening flow which evolves by translations.
In this paper we are interested in this type of solitons, which we will call {\em translating solitons (or translators)} from now on.
In general, a translating soliton is an oriented hypersurface $M$ in $\R^{n+1}$ whose  mean curvature vector field $\textbf{H}$ satisfies
$\textbf{H}=v^{\perp}$
where $v\in \mathbb{R}^{n+1}$ is a fixed vector. The vector $v$ is called the velocity. In particular, we have that the scalar mean curvature satisfies:
\begin{equation}\label{TS}
H=-\langle v,\xi \rangle,
\end{equation}
where $\xi$ is the Gauss map of $M.$ As we mentioned before, up to an intrinsic diffeomorphism 
of $M$, $M_t:=M + t v$ is a mean curvature flow. From now on, we will assume (up to dilations and
rigid motions) that the velocity of the flow is $\ee_{n+1}.$
\begin{figure}[htbp]
\begin{center}
\includegraphics[width=.78\textwidth]{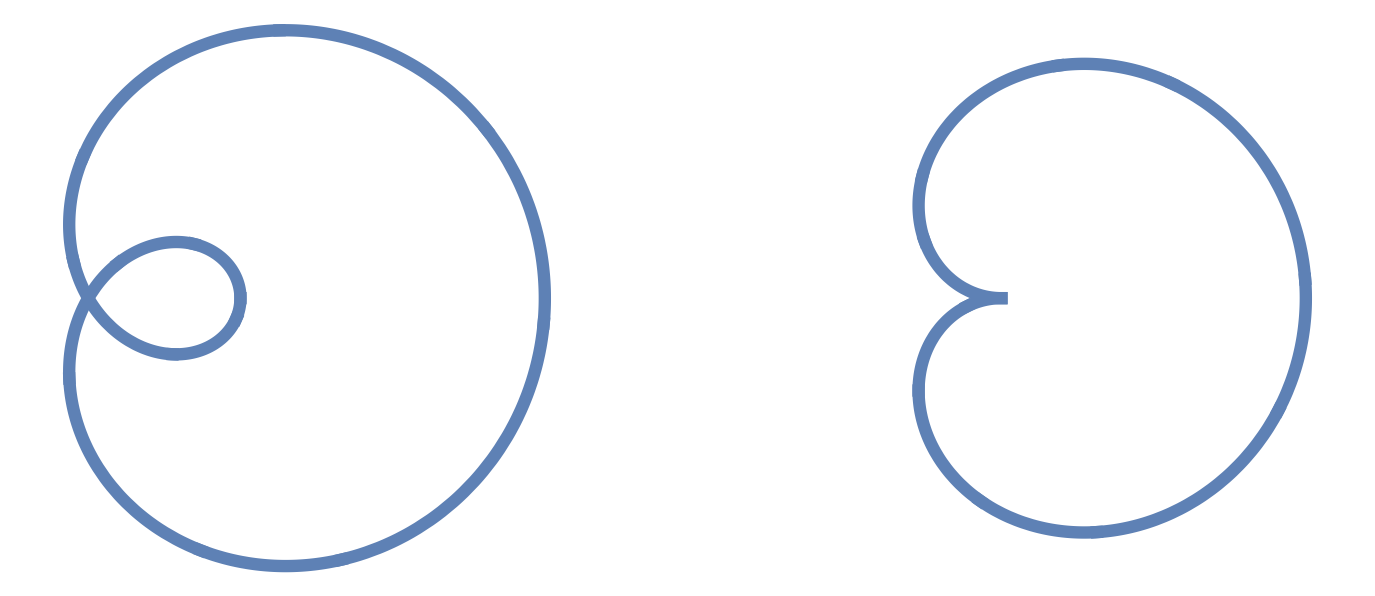}
\caption{}
\label{cardiod}
\end{center}
\end{figure}

The cylinder over a grim-reaper curve, i.e. the hypersurface in $\mathbb{R}^{n+1}$ parametrized by $F:\left(-\tfrac{\pi}{2},\tfrac{\pi}{2}\right)\times\mathbb{R}^{n-1}\longrightarrow \mathbb{R}^{n+1}$ given by $$F(x_{1},\ldots,x_{n})=(x_{1},\ldots,x_{n},-\log\cos x_{1}),$$ 
is a translating soliton, and appears as limit of sequences of parabolic rescaled solutions of mean curvature flows of immersed mean convex hypersurfaces.  For example, we can take product of the loop pinching off to a cusp times $\R^{n-1}$.  We can produce others examples of solitons just by scaling and rotating the grim reaper cylinder. In this way, we obtain a  $1-$parameter family of translating solitons parametrized by $F_\theta:\left(-\tfrac{\pi}{2\cos(\theta)},\tfrac{\pi}{2\cos(\theta)}\right)\times\mathbb{R}^{n-1}\longrightarrow \mathbb{R}^{n+1}$
\begin{equation}\label{tiltedgrim}
F_\theta(x_{1},\ldots,x_{n})=(x_{1},\ldots,x_{n},-\sec^2(\theta)\log\cos(x_{1}\cos(\theta))+\tan(\theta)x_{n}),
\end{equation}
where $\theta\in[0,\pi/2).$ Notice that the limit of the family $F_\theta$, as $\theta$ tends to $\pi/2$, is a hyperplane parallel to $\ee_{n+1}$.
\begin{figure}[htbp]
\begin{center}
\includegraphics[width=.65\textwidth]{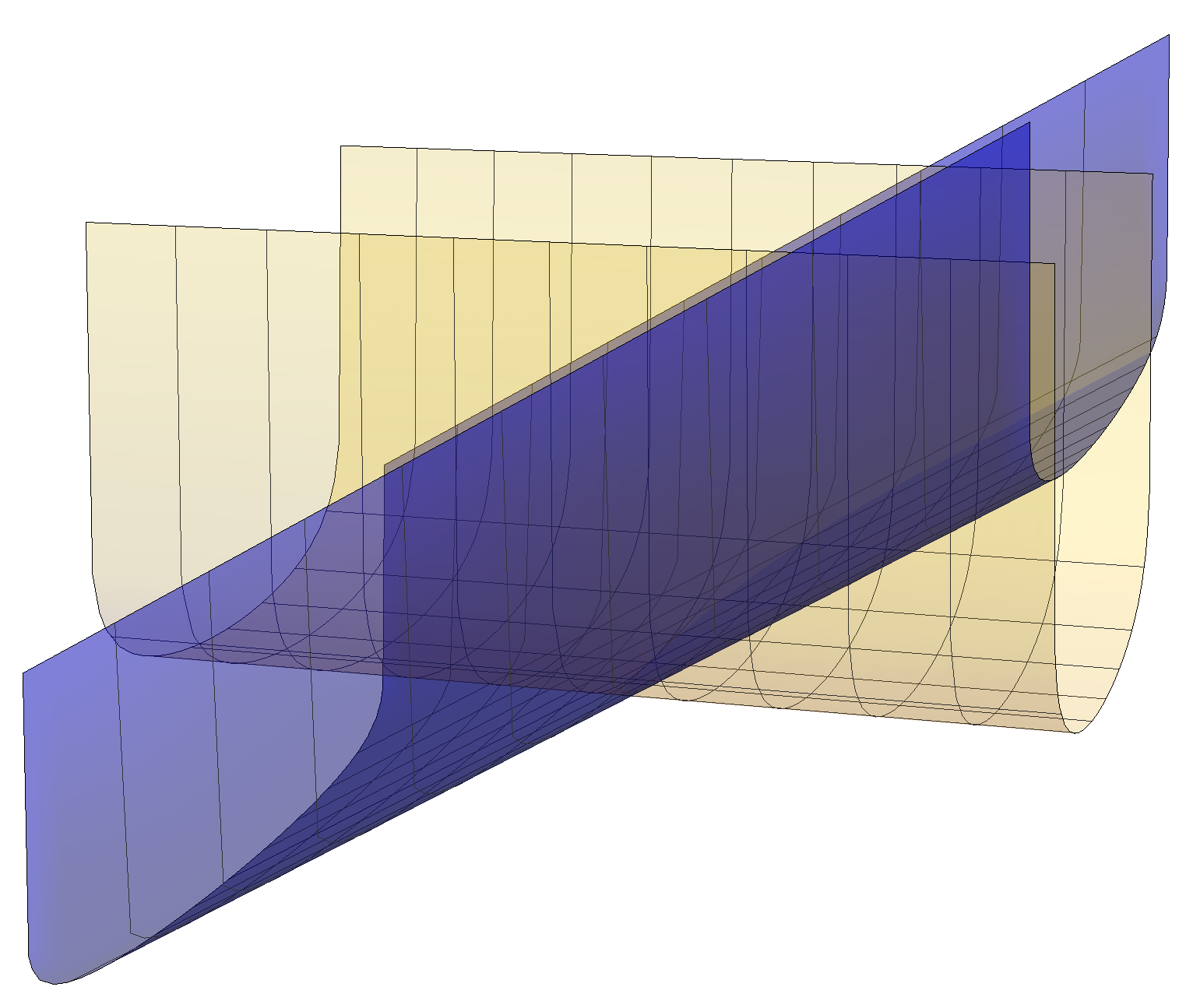}
\caption{The regular grim reaper cylinder in $\R^3$ and the tilted grim reaper for $\theta=\pi/4$.}
\label{default}
\end{center}
\end{figure}

Clutterbuck, Schn\"urer and Schulze \cite{CSS} (see also \cite{AW}) proved that there exists an entire graphical translator in $\R^{n+1}$ which is rotationally symmetric, strictly convex  with translating velocity $\ee_{n+1}$. This example is known as the translating paraboloid or bowl soliton. Moreover, they classified all the translating solitons of revolution, giving a one-parameter family $\{W^n_{\lambda}\}_{\lambda>0}$ of rotationally invariant cylinders called translating catenoids. The parameter $\lambda$ control the size of the neck of each translating soliton. The limit, as $\lambda \to 0$, of $W^n_{\lambda}$ consists of two copies of the bowl soliton with a singular point at the axis of symmetry. Furthermore, all these hypersurfaces have the following asymptotic expansion as r approaches infinity:
\begin{equation*}
\frac{r^2}{2(n-1)}-\log r+O(r^{-1}),
\end{equation*}
where $r$ is the distance function $\mathbb{R}^{n}$. 
\begin{figure}[htbp]
\begin{center}
\includegraphics[width=.55\textwidth]{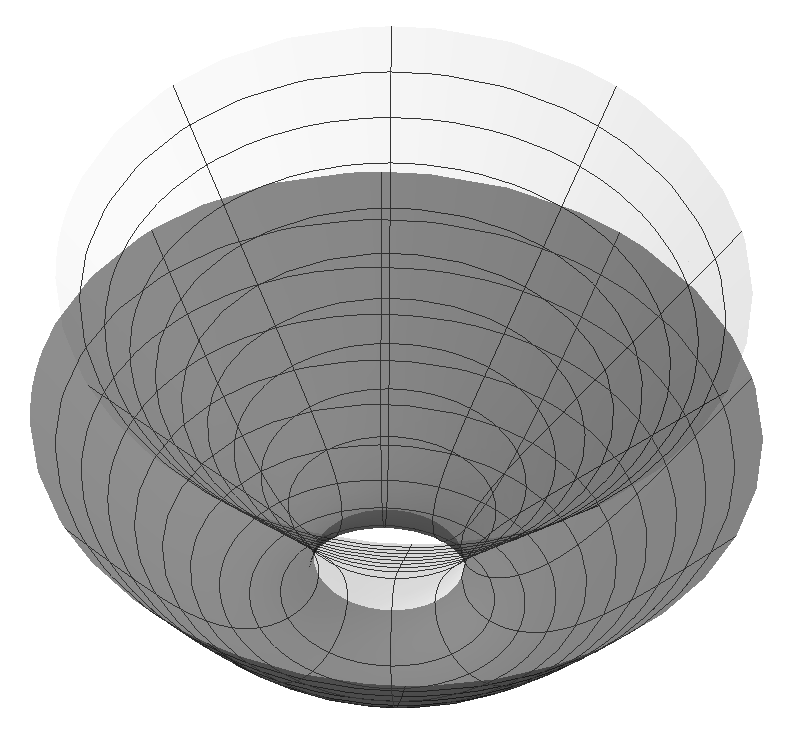}
\caption{The translating catenoid for $\lambda=2.$}
\label{catenoid}
\end{center}
\end{figure}

Recent years have witnessed the appearance of numerous examples of translators; see \cite{DDPN},\cite{Nguyen09},\cite{Nguyen13}, \cite{Nguyen15} and \cite{Smith} for more references about construction of translators. It is interesting to notice that all the known examples of complete, properly embedded, translating solitons are asymptotic to either bowl solitons or hyperplanes which are parallel to $\ee_{n+1}$.

Once  this abundance of translating solitons is guaranteed, then  arises the need to classify them. One of the first classification results was given by X.-J. Wang in \cite{Wang}. He  characterized the bowl soliton as the only convex translating soliton which is an entire graph. Very recently, J. Spruck and L. Xiao \cite{spruck-xiao} have proved that a translating soliton which is graph over the whole $\R^2$ must be convex. 

R. Haslhofer \cite{HASLHOFER} showed that any strictly convex, uniformly two-convex translator which is non-collapsing is necessarily rotationally symmetric. In this line of work, T. Bourni and M. Langford \cite{Bourni-Langford} proved that a translator which arises as a proper blow-up limit of a two-convex mean curvature flow of immersed hypersurfaces is rotationally symmetric. 
 
Using the classic Alexandrov's method of moving hyperplanes, F. Mart\'in, A. Savas-Halilaj, and K. Smoczyk \cite{MSHS15} showed that the bowl soliton is the only translating soliton that has one end and is $C^{\infty}$-asymptotic to a bowl soliton. Besides that, these authors obtained one of the first characterizations of the family of tilted grim reaper cylinders, as the only connected translation solitons in $\R^{n+1}$, $n\geq 2$, such that the function $|A|^2H^{-2}$ has a local maximum in $M\setminus H^{-1}(0).$ Some interesting classification results of grim reaper cylinders can be also found in \cite{Tasayco-zhou}.

Another characterization of the grim reaper cylinder in Euclidean $3$-space, in terms of its asymptotic behaviour, was given by F. Mart\'in, J. P\'{e}rez-Garc\'{i}a, A. Savas-Halilaj, and K. Smoczyk \cite{MPGSHS15,GARCIA16}. They proved that the grim reaper cylinder is the only connected, properly embedded, translating soliton of dimension $2$, with locally bounded genus and being $C^1$-asymptotic to two different half-planes. Their proof uses the well known fact that this kind of solitons can be seen as minimal surfaces in $\R^3$ with metric $e^{x_3} \langle \cdot, \cdot \rangle$, combined with  a compactness theorem for minimal surfaces in Euclidean $3$-manifolds due to B. White \cite{WHI15}. This compactness theorem is used for determining the asymptotic shape of the surface. Finally, the authors applied the classic versions of the maximum principle to prove that if a translating soliton is smoothly asymptotic to a grim reaper cylinder, then it must coincide with the grim reaper cylinder. 

It is not known whether White's compactness theorem has an extension for higher dimensions and, even in dimension $3$, it does not work without the hypothesis of locally bounded genus. So, the proof in \cite{MPGSHS15} fails for higher dimensions and without the hypothesis of locally bounded genus. Moreover, the tilted grim reaper cylinder given by \eqref{tiltedgrim} is $C^1-$asymptotic to two half-hyperplanes outside a non-horizontal cylinder. Hence, it is natural to ask if it is possible to generalize the theorem for arbitrary dimensions $n \geq 2$, without any further assumptions about the topology of the soliton or the axis of the cylinder. Surprisingly, the maximum principle for varifolds and the compactness theorem for stationary integral varifolds, allow us to give a positive answer to these questions.
\begin{Theorem}\label{Intro}
The hyperplanes which are parallel to $\ee_{n+1}$ and the family of tilted grim reaper cylinders are the only complete, connected, properly embedded, translating solitons in $\R^{n+1}$, $n \geq 2$, which are $C^1$-asymptotic to two half-hyperplanes outside a non-vertical solid cylinder. Moreover, if $n<7$ hyperplanes  parallel to $\ee_{n+1}$ are the only examples which are $C^1$-asymptotic to two half-hyperplanes outside a vertical cylinder.
\end{Theorem}

We would like to point out that the previous theorem is sharp in the following sense: if we remove the hypothesis about the cylinder, then the result is false. Given a vertical slab in $\R^{n+1}$ of width $w > \pi$, Hoffman, Ilmanen, White and the second author \cite{hoffman} have constructed a $(n-1)$-dimensional  family of complete graphical translators over 
this strip. When the translator that they construct is rotationally symmetric (see also \cite{bourni-2}), then these authors
also get uniqueness. From now on, we will refer these examples as $\Delta$-wings. When the $\Delta$-wing is rotationally symmetric, then it is also convex 
and it is also asymptotic to the tilted grim reaper cylinder by angle $\theta=\pm \arccos \tfrac \pi w$. These graphs are asymptotic to the vertical planes $\{x_1= \pm w/2\}$, but it is impossible (by the convexity) to find a cylinder in the hypothesis of our theorem. On the other hand, if we increase the number of asymptotic planes outside the cylinder, then  X. H. Nguyen \cite{Nguyen09, Nguyen13} also produced counterexamples.
\begin{figure}[htbp]
\begin{center}
\includegraphics[width=.55\textwidth]{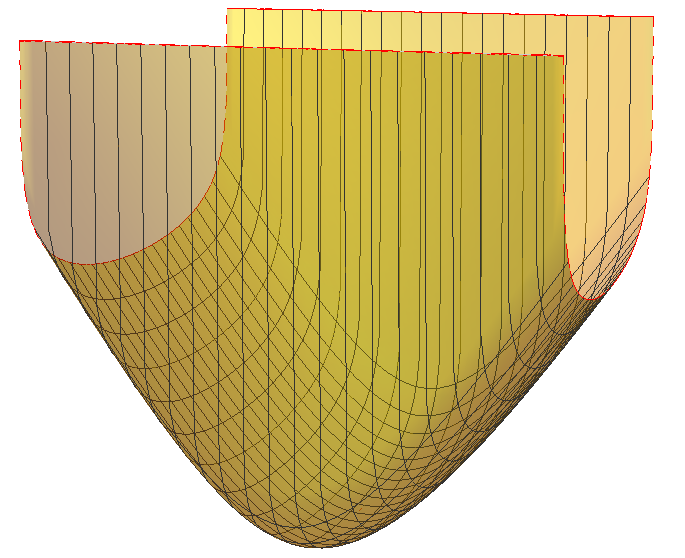}
\caption{Ilmanen's example of width $\sqrt{2} \pi.$}
\label{default-2}
\end{center}
\end{figure}

This paper is structured as follows. In Section \ref{back}, we give a short review of results from Geometric Measure Theory that we need in the paper. In Section \ref{sec:lema}, we obtain a lemma which  shows that every complete, properly embedded translating soliton in $\R^n$ satisfying the assumptions of Theorem \ref{Intro} has a surprising amount of internal dynamical periodicity. Finally, in the last section of this paper, we prove the main theorems. 

{\bf Acknowledgements.} We thank Brian White for valuable conversations and suggestions about this work. E. S. Gama is also very grateful to the Institute of Mathematics at the University of Granada for its hospitality during the time the research and preparation of this article were conducted.

\section{Background} \label{back}

In this section we will make a short review of the background we need about translating solitons and Geometric Measure Theory\footnote{A short review about varifolds can be found in the appendix of \cite{WHI}.}.
\subsection{Translating Solitons}
An oriented hypersurface $M$ in $\R^{n+1}$ satisfying 
\begin{equation*}
\textbf{H}=v^{\perp},
\end{equation*}
is called  a {\em translating soliton or translator} of the mean curvature flow. Recall that we are assuming that $v=\ee_{n+1}$, where $B=\{\ee_1,\ee_2, \ldots, \ee_{n+1}\}$ is the canonical basis of $\R^{n+1}$. The following theorem due to T. Ilmanen gives us the key to relate translating solitons with objects that are very well known and studied; minimal hypersurfaces.

\begin{Theorem}[Ilmanen \cite{ILMANEN}]\label{Ilmanen}
Translating solitons with respect to $\ee_{n+1}$ in $\mathbb{R}^{n+1}$ are minimal hypersurfaces with respect to the metric $g=e^{\frac{2}{n}x_{n+1}}\langle \cdot, \cdot \rangle,$ where $\langle \cdot, \cdot \rangle$ is the Euclidean metric.
\end{Theorem}

Now we are going to define what we mean by a hypersurface being asymptotic to half-hyperplanes outside a cylinder.

\begin{Definition}
Let $\mathcal{H}$ a open half-hyperplane in $\mathbb{R}^{n+1}$ and $w$ the unit inward pointing normal of $\partial \mathcal{H}$. For a fixed positive number $\delta$, denote by $\mathcal{H}(\delta)$ the set given by
\begin{equation*}
\mathcal{H}(\delta):=\left\{p+tw:p\in \partial \mathcal{H}\ \operatorname{and}\ t>\delta\right\}.
\end{equation*}
We say that a smooth hypersurface $M$ is $C^{k}-$asymptotic to the open half-hyperplane $\mathcal{H}$ if $M$ can be represented as the graph of a $C^{k}-$ function $\varphi:\mathcal{H}\longrightarrow \mathbb{R}$ such that for every $\epsilon>0$, there exists $\delta>0$, so that for any $j\in\{1,2,\ldots,k\}$ it holds
\begin{equation*}
\sup_{\mathcal{H}(\delta)}|\varphi|<\epsilon\ \operatorname{and}\ \sup_{\mathcal{H}(\delta)}|D^j\varphi|<\epsilon.
\end{equation*}
We will say that a smooth hypersurface $M$ is $C^{k}-$asymptotic outside a cylinder to two half-hyperplanes $\mathcal{H}_{1}$ and $\mathcal{H}_{2}$ if there exists a solid cylinder $\mathcal{C}$ such that:
\begin{itemize}
\item[i.]The solid cylinder $\mathcal{C}$ contains the boundaries of the half-hyperplane $\mathcal{H}_{1}$ and $\mathcal{H}_{2}$,
\item[ii.]$M\setminus \mathcal{C}$ consists of two connected components $M_{1}$ and $M_{2}$ that are  $C^{k}-$asympto-tic to $\mathcal{H}_{1}$ and $\mathcal{H}_{2}$, respectively.
\end{itemize}
\end{Definition}

\begin{Remark}
Notice that the solid cylinders in $\R^{n+1}$ that we are considering are those whose boundary is isometric to $\mathbb{S}^{1}(r)\times\R^{n-1},$ where $\mathbb{S}^{1}(r)$ is the sphere of radius $r$.
\end{Remark}
Simple examples of this case are the hyperplanes parallel to $\ee_{n+1}$. Others examples are given by the family of tilted grim reaper cylinders, which are $C^{\infty}-$asymptotic to two half-hyperplanes outside the corresponding tilted cylinder. 
 
\subsection{Maximum Principle}

In this short section, we are going to give a short review of the maximum principle for varifolds. Before stating it, let us recall the classical version of the maximum principle.

\begin{Theorem}[Interior Maximum Principle]\label{Interior Maximum Principle}
Let $M_{1}$ and $M_{2}$ be $n-$dimensional embedded translating solitons, with not necessarily empty boundaries, in the Euclidean space $\mathbb{R}^{n+1}.$ Suppose that there exists a common point $x$ in the interior of $M_{1}$ and $M_{2}$ where the corresponding tangent spaces coincide and that $M_{1}$ lies at one side of $M_{2}$. Then $M_{1}=M_{2}$.\
\end{Theorem}

Although this maximum principle has very useful applications, it requires a very strong hypothesis: smoothness. As we will work with varifolds, we will need another version of this theorem that can be applied in this setting. The version that we will use in this paper is due to B. Solomon and B. White in \cite{SW} ( see also \cite{WHI09} Theorem 4). Recall that a varifold $V$ in a Riemannian manifold $N$ minimizes area to first order if the first variation of $V$ , denoted by $\delta V$, satisfy
\begin{equation*}
\delta V(X)\geq 0,
\end{equation*}
for every compactly supported $C^{1}$ tangent vector field $X$ on $N$ satisfying
\begin{equation*}
\langle X,\nu_{\partial N} \rangle\geq 0
\end{equation*}
at all point of $\partial N,$ where $\nu_{\partial N}$ is the unit inward pointing normal of $\partial N$. After this short review, we can state the one of the main tools of this paper.

\begin{Theorem}[Strong Maximum Principle for Minimal Hypersurfaces \cite{SW}]\label{Strong Maximum Principle}
Suppose that $N^{m+1}$ is a Riemannian manifold, not necessarily complete, with connected non-empty boundary $\partial N$, and that $N^{m+1}$ is mean convex, i.e. that
\begin{equation*}
    \langle \mathbf{H},\nu_{\partial N} \rangle\geq0
\end{equation*}
on $\partial N$ where $\mathbf{H}$ is the mean curvature vector of $\partial N$ and where $\nu$ is the unit inward pointing normal of $\partial N$. Let $V$ be a $m$-dimensional varifold that minimizes area to first order in $N^{m+1}$. If $\spt V$ contains point of $\partial N$, then it must contain all $\partial N$.
\end{Theorem}

\subsection{Compactness Theorems for Varifolds} \label{compactness}

The purpose of this section consists of reviewing the compactness theorems for varifolds that we will need along the paper. These theorems combined with the strong maximum principle for minimal hypersurfaces will be our main tools. Let us start with the definitions of {\em weak convergence} and {\em smooth convergence}. Further details can be found in \cite{ALLARD,SIMON,Schoen-Simon}.

\begin{Definition}\label{Weakly Convergence}
Let $\left\{V_{i}\right\}_{i\in\mathbb{N}}$ be a sequence of varifolds in a smooth manifold $N$. We say that $\{V_{i}\}$ converges weakly to the varifold $V$, if for every smooth function $\varphi:N\longrightarrow\mathbb{R}$ with compact support we have
\begin{equation*}
\lim_{i\longrightarrow \infty}\int_{N}\varphi d\mu_{V_{i}}=\int_{N}\varphi d\mu_{V},
\end{equation*}
where $d\mu_{V_{i}}$ is the Radon measure in $N$ associated to the varifold $V_{i}.$ If the sequence $\left\{V_{i}\right\}$  converges weakly to $V$, then we will write $V_{i}\rightharpoonup V$.
\end{Definition}
Now let $M$ be a hypersurface in a Riemannian manifold $N$. Given $p\in M$ and $r>0$ we denote by 
\begin{equation*}
B_{r}(p)=\left\{v\in T_{p}M\; : \;|v|<r\right\}
\end{equation*}
the tangent ball around $p$ of radius $r$. Consider now $T_{p}M$ as a vector subspace of $T_{p}N$ and let $\nu$ be a unit normal vector to $T_{p}M$ in $T_{p}N.$ Fix a sufficiently small $\epsilon>0$ and denote by $W_{r,\epsilon}(p)$ the solid cylinder around $p$, that is
\begin{equation*}
W_{r,\epsilon}(p):=\left\{\exp_p(q+t\nu)\; : \;q\in B_{r}(p)\ \operatorname{and}\ |t|<\epsilon\right\},
\end{equation*}
where  $\exp$ is the exponential map of the Riemannian manifold $N.$ Given a smooth function $f:B_{r}(p)\longrightarrow\mathbb{R},$ the set
\begin{equation*}
Graph(f):=\left\{\exp_p(q+f(q)\nu)\; : \;q\in B_{r}(p)\right\},
\end{equation*}
is called the graph of $f$ over $ B_{r}(p).$ \

Now, we can define the convergence in $C^{\infty}-$topology.

\begin{Definition}\label{Smooth Convergence}
Let $\left\{M_{i}\right\}$ be a sequence of hypersurfaces in a smooth manifold $N$. We say that $\{M_{i}\}$ converges in $C^{\infty}-$ topology with finite multiplicity to a smooth embedded hypersurface $M$ if \
\begin{itemize}
\item[a.] $M$ consists of accumulations points of $\{M_{i}\}$, that is, for each $p\in M$ there exists a sequence $\{p_{i}\}$ such that $p_{i}\in M_{i}$, for each $i\in \mathbb{N}$, and $p=\textstyle\lim_{i}p_{i}$;

\item[b.] For every $p\in M$ there exists $r,\epsilon>0$ such that $M\cap W_{r,\epsilon}(p)$ can be represented as the graph of a function $f$ over $B_{r}(p)$;

\item[c.] For $i$ large enough, the set $M_{i}\cap W_{r,\epsilon}(p)$ consists of a finite number, $k$, independent of $i$, of graphs of functions $f_{i}^1,\ldots,f^{k}_{i}$ over $B_{r}(p)$ which convergence smoothly to $f.$\
\end{itemize}
The  multiplicity of a given point $p\in M$ is defined by $k$. As $\left\{M_{i}\right\}$ converges smoothly to $M$, then we will write $M_{i}\longrightarrow M$.
\end{Definition}

\begin{Remark}
In general, the limit submanifold is not necessarily connected.
\end{Remark}
At this point, we can state the weak version of the compactness theorem ( see \cite[section 6.4]{ALLARD} or \cite[Theorem 42.7]{SIMON}).

\begin{Theorem}[Compactness Theorem for Integral Varifold \cite{ALLARD}]\label{Compactness Theorem for Integral Varifold}
Let $\left\{M_{i}\right\}$ be a sequence of minimal hypersurfaces in $\mathbb{R}^{n+1}$, with not necessary the canonical metric, whose area is locally bounded, then a subsequence of $\left\{M_{i}\right\}$ converge weakly to a stationary integral varifold $M_{\infty}.$
\end{Theorem}

\begin{Remark}
Notice that in the previous theorem we may have \(M_\infty=\varnothing\). Indeed, if $M_i:=\{x\in\R^{n+1}\;:\; \langle x,\ee_{n+1}\rangle=i\}$, then $\{M_i\}$  is a sequence of minimal hypersurfaces in $\R^{n+1}$ with the Euclidean metric whose area is locally bounded and $M_i\rightharpoonup \varnothing.$
\end{Remark}

Actually, when we know that each $M_{i}$ is stable, then  we can conclude that the convergence above is strong, in the sense that it is smooth away from a closed set, called the {\em singular set}. This is the content of the next theorem (see \cite[Theorem 3]{Schoen-Simon} and \cite[Theorem 18.1]{Wickramasekera}).

\begin{Theorem}[Strong Compactness Theorem \cite{Schoen-Simon}]\label{Strong Compactness Theorem}
Let $\left\{M_{i}\right\}$ be a sequence of stable minimal hypersurfaces in $\mathbb{R}^{n+1}$, with not necessary the canonical metric, with locally bounded area. Then there exist a closed set, Sing, and a stationary integral varifolds, $M_{\infty}$, such that a subsequence of $\left\{M_{i}\right\}$ converges smoothly to $M_{\infty}$ away from Sing. Moreover, the set Sing has Haussdorff dimension at most $n-7$. Hence it is empty for $n<7$.
\end{Theorem}

As in \cite{MPGSHS15}, we will regard all translating solitons as minimal hypersurfaces in $\R^{n+1}$ with Ilmanen's metric. Then we use the maximum principle and compactness theorems to get a point of contact with a given smooth hypersurface, in our case either hyperplanes parallel to $\ee_{n+1}$ or an element in the family of tilted grim reaper cylinders. In our search for points of contact, it will be interesting to know if a given translator is stable. Hence, we can use the previous theorem to get smooth convergence and to obtain more information about our hypersurface. Similarly to what happens to minimal hypersurfaces in $\mathbb{R}^{n+1}$ with the canonical metric, it is possible show that graphs are stable. This simple criterion is due to L. Shahriyari for $n=2$.

\begin{Theorem}[Shahriyari  \cite{SHAHRIYARI15}]\label{Shahriyari-Xin}
Let $M$ be a translator hypersurface which is a graph over a hyperplane. Then, as minimal graphical hypersurface in $\mathbb{R}^{n+1}$ with Ilmanen's metric, is stable. 
\end{Theorem}
\begin{Remark}
Although Shahriyari proved the above result for $n=2$, it is straightforward to check that the same proof works for arbitrary dimensions. Y. L. Xin \cite{XIN16} proved that if $M$ is a vertical graph, then it is area-minimizing as a hypersurface in $\mathbb{R}^{n+1}$ with Ilmanen's metric.
\end{Remark}

Another theorem that we will need in this paper is a version of Allard's regularity theorem. This theorem gives us a powerful criterion to decide when the weak convergence is in fact smooth. This version of the Allard's regularity theorem due to White \cite{WHI16}.

\begin{Theorem}[Allard's Regularity Theorem]\label{Allard's Regularity Theorem}
Let $\left\{M_{i}\right\}$ be a sequence of properly embedded minimal hypersurfaces without boundary in $\mathbb{R}^{n+1}$, with not necessarily the canonical metric. Suppose that sequence $\left\{M_{i}\right\}$ in $\mathbb{R}^{n+1}$ converges weakly to $S\subset M\subset\mathbb{R}^{n+1}$, where $M$ is a connected smoothly embedded hypersurface, and some point in $M$ has a neighbourhood $U \subset \mathbb{R}^{n+1}$ such that $\{M_{i} \cap U\}$ converges weakly to $M \cap U$ with multiplicity one, then $\{M_{i}\}$ converges to $M$ smoothly and with multiplicity one everywhere.
\end{Theorem}

\section{The Dynamics Lemma} \label{sec:lema}
Fix $\theta \in [0, \pi/2)$ and define $\displaystyle{u_{\theta}:=-\sin (\theta) \cdot  \ee_n+\cos(\theta) \cdot \ee_{n+1}}$. For a given $r>0$, we consider the cylinder $$\displaystyle{\mathcal{C}:=\{x=(x_1, \ldots,x_{n+1}) \in \mathbb{R}^{n+1} \; : \; x_1^2+\langle u_\theta,x\rangle^2 \leq r^2\}}.$$
Throughout this section $M^n$ will be a complete, connected, properly embedded translating soliton in $\mathbb{R}^{n+1}$ such that, outside $\mathcal{C}$, $M$ is $C^{1}$--asymptotic to two half-hyperplanes $\mathcal{H}_{1}$ and $\mathcal{H}_{2}$.
In this section, we are going to obtain our main lemma which shows that such a soliton has a surprising amount of internal dynamical periodicity. As a consequence, we shall deduce several properties about the asymptotic geometry of $M$.

\begin{Lemma}[Dynamics Lemma]\label{Compactness Lemma}
Let $M$ be a hypersurface as above. Suppose that $\left\{b_{i}\right\}_{i\in \mathbb{N}}$ is a sequence in $[\ee_1,u_{\theta}]^\perp$ and let $\left\{M_{i}\right\}_{i\in \mathbb{N}}$ be a sequence of hypersurfaces given by $M_{i}:=M+b_{i}.$ Then, after passing to a subsequence, $\{M_{i}\}$ converges weakly to a connected stationary integral varifold $M_{\infty}$. Moreover, $M_{\infty}$ is smooth outside the cylinder $\mathcal{C}$ and away from a singular set of Hausdorff dimension at most $n-7$, and the convergence is smooth outside the cylinder and away from the singular set.
\end{Lemma}
\begin{proof} The strategy of the proof follows a similar argument as in \cite{MPGSHS15}. However, as we do not assume any hypotheses about the topology (inside the cylinder) besides the fact that we are working in arbitrary dimensions, then we have to use the Theorem \ref{Compactness Theorem for Integral Varifold} to overcome the difficulties. 

From our assumption on $M$, there exist $r>0$ and two half-hyperplanes, $\mathcal{H}_{1}$ and $\mathcal{H}_{2}$, such that the connected components of $M$ outside $\mathcal{C}$ are $C^1-$asymptotic to them. Let $w_{1}$ and $w_{2}$ are the unit inward normal vectors of $\partial\mathcal{H}_{1}$ and $\partial\mathcal{H}_{2}$, respectively. For every $\delta>0$ consider the closed half-hyperplanes 
$\mathcal{H}_{k}(\delta):=\{p+tw_{k}\; : \;p\in \partial\mathcal{H}_{k}$ and  $t\geq\delta \},$
$k=1,2$. 
\begin{figure}[htpb]
\begin{center}
\includegraphics[width=.65\textwidth]{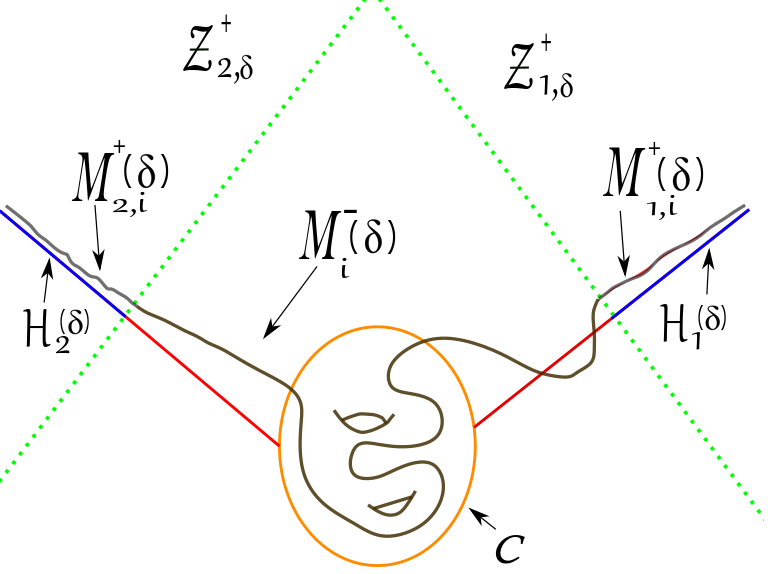}
\end{center}
\caption{\small Transversal section of the behaviour of $M_{i}$.}\label{figure 1}
\end{figure}
For $k=1,2$, let $\mathcal{Z}^{+}_{k,\delta}$ denote the half-space in $\mathbb{R}^{n+1}$ which contains $\mathcal{H}_{1+k}(\delta)$ and whose boundary contains $\partial\mathcal{H}_{1+k}(\delta)$ and is perpendicular to $w_{k}$. We also call
$
\mathcal{Z}^{-}_{k,\delta}:=(\mathbb{R}^{n+1}-\mathcal{Z}^{+}_{k,\delta})\cup\partial\mathcal{Z}^{+}_{k,\delta}
,$
for $k=1,2$.

Now consider the sets $M_{k,i}^{+}(\delta):=M_{i}\cap\mathcal{Z}^{+}_{k,\delta}$ and $M_{i}^{-}(\delta):=M_{i}\cap\mathcal{Z}^{-}_{1,\delta}\cap\mathcal{Z}^{-}_{2,\delta}$,  for all $k\in\{1,2\}$.  In order to apply the compactness results in \S \ref{compactness}, we must deduce that the sequences $\left\{M_{k,i}^{+}(\delta)\right\}$ and $\left\{M_{i}^{-}(\delta)\right\}$ have locally bounded area. This can be done by using  similar ideas as in \cite[Lemma 3.1]{MPGSHS15}. Let us explain something about the proof of this fact. First for all, we should show that $\left\{M_{k,i}^{+}(\delta)\right\}$ has locally bounded area for a sufficiently large $\delta>0$. This can be done by using the fact that $M_{k,i}^{+}(\delta)$ can be represented as graph of a function defined over a half-hyperplane. Now, we fix a sufficiently large $\delta$, using the fact that the boundary of $M_{i}^{-}(\delta)$ can be represented as graph, and obtain that  the boundary of the sequence $\{M_{i}^{-}(\delta)\}$ has locally bounded area. Now the area blow-up set defined by 
\begin{equation}\label{blowup}
\mathcal{U}:=\{x\in\R^{n+1}\;:\limsup\operatorname{Area}(M_{i}^{-}(\delta)\cap B(p,r))=\infty \operatorname{for\ every} r>0 \}
\end{equation}
lies inside a large non-horizontal cylinder, where $B(p,r)$ is the geodesic ball in $\R^{n+1}$ with Ilmanen's metric. We need to prove that $\mathcal{U}=\varnothing$ to get that the sequence $\{M_{i}^{-}(\delta)\}$ has locally bounded area too. Arguing by contradiction, let us suppose that $\mathcal{U}\neq\varnothing$.  In this case, we could take a tilted grim reaper cylinder whose axis is perpendicular to $\mathcal{C}$ and it does not intersect $\mathcal{C}$. Now we could move the tilted grim reaper cylinder until we get a first point of contact with $\mathcal{U}$, but the Strong Barrier Principle \cite[Theorems 2.6 and 7.3]{WHI15} says that $\mathcal{U}$ must be the tilted grim reaper cylinder, which is absurd.

As the sequence $\left\{M_{i}\right\}$ has locally bounded area, by Theorem \ref{Compactness Theorem for Integral Varifold} there exists a subsequence of $\left\{M_{i}\right\}$, which we still denote by  $\left\{M_{i}\right\}$, that converges weakly to the stationary integral varifold $M_{\infty}$. Furthermore, as outside $\mathcal{C}$ both connected components of $M_{i}$ are graphs (in particular stable by Theorem \ref{Shahriyari-Xin}), we can apply Theorem \ref{Strong Compactness Theorem} to conclude that the convergence is smooth outside $\mathcal{C}$ and away from a singular set with dimension at most $n-7$. This implies that $M_{\infty}$ is smooth outside $\mathcal{C}$ and away from the singular set.\ 

Using this last fact, we can conclude the connectedness of $M_{\infty}$ as follows. Taking into account that any loop in $\mathbb{R}^{n+1}$ intersects $M_{\infty}$ in an even numbers of points (counting multiplicity), then both wings of $M_\infty$ must lie in the same connected component. Indeed, if this was not true, then we could choose the above mentioned loop intersecting $M_{\infty}$ at one unique point (because $M_{\infty}$ is smooth outside a sufficiently large cylinder and away from singular set) which is absurd. This implies that if $M_{\infty}$ is not connected, there would be a connected component inside the cylinder. In this case we can consider a suitable tilted grim reaper (whose axis is perpendicular to $u_{\theta}$) of sufficiently large coordinate in the direction of $u_{\theta}$ so that it is not intersect the solid cylinder. Now, if we move it in the direction of $-u_{\theta}$ until it touchs the component inside the cylinder at a first point of contact, then we get a contraction because the component inside the cylinder must be the whole tilted grim reaper by Theorem \ref{Strong Maximum Principle}.  Hence $M_{\infty}$ is connected.
\end{proof}

\begin{Remark} \label{re:20}
Notice that the proof above shows a stronger property. Suppose $M_{i}\rightharpoonup M_{\infty},$ where $\{M_{i}\}$ and $M_{\infty}$ are as in the previous lemma. Take any sequence $\left\{c_{j}\right\}_{j\in \mathbb{N}}$ in the subspace $[\ee_1,u_\theta]^\perp$ and define the sequence $\left\{\mathfrak{M}_{j}:=M_{\infty}+c_{j}\right\}$ of connected stationary integral varifolds. Reasoning as above, we have that the sequence $\left\{\mathfrak{M}_{j}\right\}$ has locally bounded area on the compacts sets and the singular set of each $\mathfrak{M}_{j}$ has Hausdorff dimension at most $n-7$ . Therefore by \cite[Theorem 3]{Schoen-Simon} or \cite[Theorem 18.1]{Wickramasekera} we can assume, up to a subsequence, that $\mathfrak{M}_{j}\rightharpoonup \mathfrak{M}_\infty,$ where $\mathfrak{M}_\infty$ is a stationary integral varifold whose singular set has Hausdorff dimension at most $n-7$ and away from the singular it is stable hypersurface. Then, we can apply the argument above to conclude that $\mathfrak{M}_{\infty}$ satisfies the same properties as $M_{\infty}$.
\end{Remark}

Now we are going to use the previous lemma to conclude that $w_1$ and $w_2$ are parallel to $u_\theta$. Moreover, if the half-hyperplanes $\mathcal{H}_1$  and $\mathcal{H}_2$ are parts of the same hyperplane, then $M$ is a hyperplane parallel to $\ee_{n+1}$. 
\begin{Lemma}\label{Characterization of the Hyperplane}
Let $M$ be a hypersurface as above. Then, the half-hyperplanes $\mathcal{H}_{1}$ and $\mathcal{H}_{2}$ must be parallel to $\ee_{n+1}$. Moreover, if $\mathcal{H}_{1}$ and $\mathcal{H}_{2}$ are parts of the same hyperplane $\Pi$, then $M$ must coincide with $\Pi.$
\end{Lemma}
\begin{proof}We will proceed by contraction. Assume to the contrary that the half-hyperplane
\begin{equation*}
\displaystyle{\mathcal{H}_{1}=\left\{p+tw_{1}\; : \; p \in \partial \mathcal{H}_{1},\ t>0\right\}}
\end{equation*}
is not parallel to direction of translation $\ee_{n+1}$. Notice that $\ee_{j}$ and $E_{n}$ are perpendicular to $w_{1}$ for all $j\in\{2,\ldots,n-1\},$ where $E_{n}:=\cos(\theta)\ee_{n}+\sin(\theta)\ee_{n+1}$. Therefore $w_{1}$ is not parallel to $u_{\theta}.$ In this case, $w_{1}$ form a non-vanishing angle only with $\ee_{1}$, that we denote by $\alpha:= \measuredangle(\ee_1,w_1)$. Suppose that $\cos \alpha>0$. For given real numbers $t$ and $l$, we consider the tilted grim reaper cylinder:
\begin{equation*}
\displaystyle{\mathcal{G}^{t,l}:=\left\{ F_\theta(x_1-t,x_2, \ldots,x_n)+t \ee_1+lu_{\theta} \; : \;
|x_{1}-t|<\tfrac{\pi}{2\cos(\theta)} ,\  (x_{2},\ldots,x_{n})\in \mathbb{R}^{n-1}
\right\}}.
\end{equation*}
By our assumptions about $M$, if $\delta$ is sufficiently large then $M^{+}_{1}(\delta):=M\cap\mathcal{Z}^{+}_{1,\delta}$ is sufficiently close to $\mathcal{H}_{1}$. From this we can conclude that there exist sufficiently large $t_{0},l_{0}\in \mathbb{R}$ so that $\mathcal{G}^{t_{0},l_{0}}$ do not intersect $M_{1}^{+}(\delta)$ (see Figure \ref{figure 2}). In fact, we can choose $t_{0}$ so that $\partial M_{1}^{+}(\delta)\cap S_{t_{0}}=\varnothing,$ where $\displaystyle{S_{t_{0}}=\left(t_{0}-\tfrac{\pi}{2\cos(\theta)},t_{0}+\tfrac{\pi}{2\cos(\theta)}\right)\times \mathbb{R}^{n}}.$ Since $\mathcal{H}_{1}$ is not parallel to $\ee_{n+1}$ if we translate $\mathcal{G}^{t_{0},l_{0}}$ at direction of $-u_{\theta}$ we conclude that there exists a first $l_{1}$ such that either $\mathcal{G}^{t_{0},l_{0}-l_{1}}$ and $M_{1}^{+}(\delta)$ have a point of contact or $\dist\left(\mathcal{G}^{t_{0},l_{0}-l_{1}},M_{1}^{+}(\delta)\right)=0.$
\begin{figure}[htpb]
\begin{center}
\includegraphics[width=.48\textwidth]{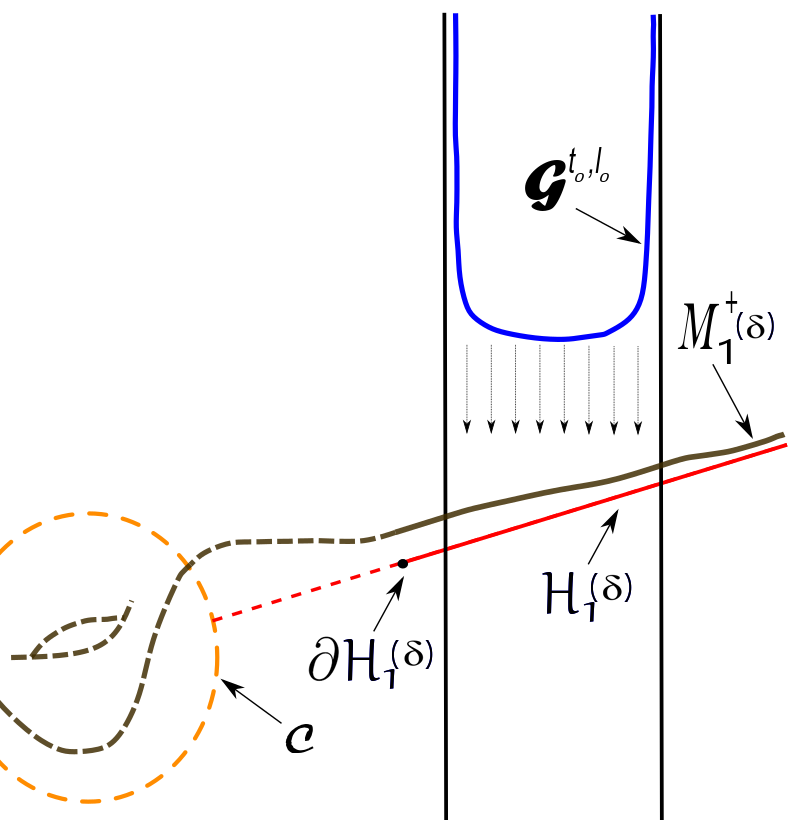} 
\end{center}
\caption{Transversal section of the behaviour of $M^{+}_{1}(\delta)$ and $\mathcal{G}^{t_{0},l_{0}}$.} \label{figure 2}
\end{figure}

According to Theorem \ref{Strong Maximum Principle} the first case cannot be possible because of our assumptions on $M.$ On the other hand, the second case implies that there exists a sequence $\left\{p_{i}=(p_{i}^{1},\ldots,p_{i}^{n+1})\right\}$ in $M_{1}^{+}(\delta)$ such that: 
\begin{itemize}
\item[a.]  The sequence $\left\{\langle p_{i}^{n+1},u_{\theta}\rangle\right\}$ is bounded in $\R$;
\item[b.] $\textstyle{\lim_{i}\dist\left(\mathcal{G}^{t_{0},l_{0}-l_{1}},p_{i}\right)=0}.$
\end{itemize}
Note that the sequence $\left\{p_{i}^{1}\right\}$ is bounded (by the asymptotic behaviour of $\mathcal{G}^{t_{0},l_{0}-l_{1}}$). Thus, up to a subsequence, we can suppose $\{p_{i}^{1}\}\to p_{\infty}^{1}$ and $\{\langle p_{i},u_{\theta}\rangle\}\to p_{\infty}^{u_{\theta}}$. Let $\displaystyle{\left\{M_{i}:=M-(0,p_{i}^2,\cdots,p_{i}^{n+1})+\langle p_{i},u_{\theta}\rangle u_{\theta}\right\}}$ be a sequence of hypersurfaces in $\mathbb{R}^{n+1}.$ By Lemma \ref{Compactness Lemma}, we can suppose that $M_{i}\rightharpoonup M_{\infty}$, where $M_{\infty}$ is a connected stationary integral varifold. We claim that $\displaystyle{p_{\infty}=p_{\infty}^{1}\ee_{1}+p^{u_{\theta}}_{\infty}u_{\theta}\in \spt M_{\infty}}$. Indeed, take any small open ball $\B_{r}(p_{\infty})$ of radius $r$ around $p_{\infty}$ in $\mathbb{R}^{n+1}$ and let $p_{i}^{*}=p_{i}^{1}\ee_{1}+\langle p_{i},u_{\theta}\rangle u_{\theta}.$ Since $\{p_{i}^{*}\}\to p_{\infty}$, we obtain that there is an $i_{0}$ such that $\dist_{g}\left(p_{i}^{*},p_{\infty}\right)<\frac{r}{4}.$ Now fix any $\epsilon\in \left(0,\frac{r}{8}\right).$ Using the definition of limit of variolfd (see \cite{ALLARD}, Section 2.6), we have 
\begin{equation*}
\lim_{i}\mu_{M_{i}}(\B_{r}(p_{\infty}))=\mu_{M_{\infty}}(\B_{r}(p_{\infty})),
\end{equation*}
where $\mu_{M_{i}}$ is the Radon (Riemannian) measure in $\mathbb{R}^{n+1}$ associated to $M_{i}$. On the other hand, if $i>i_{0}$ we have $\displaystyle{\mu_{M_{i}}(\B_{r}(p_{\infty}))\geq \mu_{M_{i}}(\B_{\epsilon}(p_{i}^{*}))\geq c(n,\epsilon)}$ by the monotonicity formula for varifold, where $c(n,\epsilon)$ is a positive constant that depends only on $n$ and $\epsilon,$ because each $M_{i}$ is a smooth hypersurface. In particular, we have
\begin{equation*}
\mu_{M_{\infty}}(\B_{r}(p_{\infty}))\geq c(n,\epsilon)>0,
\end{equation*}
that is, $p_{\infty} \in \spt M_{\infty}$, and follows that $\spt M_{\infty}$ and $\mathcal{G}^{t_{0},l_{0}-l_{1}}$ have a point of contact at $p_{\infty}.$  Therefore, by Theorem \ref{Strong Maximum Principle} we must have $M_{\infty}=\mathcal{G}^{t_{0},l_{0}-l_{1}}.$ But this is clearly impossible by our assumption about $w_{1}$ and Theorem \ref{Allard's Regularity Theorem}. Analogously, we can conclude that $\cos \alpha$ cannot be negative and that $\mathcal{H}_{2}$ is parallel to $\ee_{n+1}.$

Finally, if $\mathcal{H}_{1}$ and $\mathcal{H}_{2}$ are part of the same hyperplane $\Pi$, which we suppose to be $[\ee_{1}]^{\perp}$, then we claim that the first coordinate must be constant on $M$. In fact, suppose to the contrary that this is true. In this case, the first coordinate $x_{1}$ takes a extreme value either at point in $M$ or along of a sequence $\left\{p_{i}=(p_{i}^{1},\ldots,p_{i}^{n+1})\right\}$ such that $\{\langle p_{i},u_{\theta}\rangle\}\to p_{\infty}^{u_{\theta}}$. From Theorem \ref{Strong Maximum Principle} the first case is impossible. Regarding the second case, suppose that $\{x_{1}(p_i)\}\to \sup_{M} x_1(>0)$ and  let us denote $\displaystyle{\Pi_{1}:=\sup_M x_{1}\ee_{1}+\operatorname{span}[\ee_{2},\ldots, \ee_{n}, u_{\theta}]}$. Again, if we consider the sequence 
\begin{equation*}
\left\{M_{i}:=M-(0,p_{i}^{2},\ldots,p_{i}^{n+1})+\langle p_{i},u_{\theta}\rangle u_{\theta}\right\},
\end{equation*}
by Lemma \ref{Compactness Lemma}, a subsequence converges to $M_{\infty}$, where $M_{\infty}$ is a connected stationary integral varifold, thus (reasoning as above) we have a interior point of contact between $\spt M_{\infty}$ and $\Pi_{1}$. So, by Theorem \ref{Strong Maximum Principle} we conclude that $M_{\infty}=\Pi_{1},$ which is impossible. This shows that the first coordinate $x_{1}$ is constant.  Therefore $M$ must be the hyperplane $\Pi.$ 
\end{proof}

Let us finish this section with another application of Lemma \ref{Compactness Lemma}. 

\begin{Lemma}\label{Maximum Principle for Unbounded Domain}
Let $M$ be a hypersurface satisfying the previous conditions and assume that the half-hyperplanes $\mathcal{H}_{1}$ and $\mathcal{H}_{2}$ are distinct. Consider a domain $\Sigma$ of $M$ (not necessarily compact) with non-empty boundary $\partial\Sigma$ such that the function $x\mapsto\langle x,u_{\theta}\rangle$ of $\Sigma$ is bounded. Then the supremum and the infimum of the $x_{1}-$coordinate function of $\Sigma$ are reached along the boundary of $\Sigma$ i.e., there exists no sequence $\left\{p_{i}\right\}$ in the interior of $\Sigma$ such that $\displaystyle{\lim_{i\rightarrow\infty}\dist\left(p_{i},\partial\Sigma\right)>0}$ and either $\displaystyle{\lim_{i\rightarrow\infty}x_{1}(p_{i})=\sup_{\Sigma}x_{1}}$ or $\displaystyle{\lim_{i\rightarrow\infty}x_{1}(p_{i})=\inf_{\Sigma}x_{1}}$.
\end{Lemma}
\begin{proof}
The proof of this Lemma follows the same ideas as in \cite{MPGSHS15}. The only difference is that in the proof we should use the Lemma \ref{Compactness Lemma} and Theorem \ref{Strong Maximum Principle} to conclude.
\end{proof}

\section{The proof of the main theorems} 

\subsection{The case $\theta \in [0, \pi/2)$}
This section is devoted to demonstrating the main theorem for $\theta \in [0, \pi/2)$. As in the previous section, we fix $\textstyle{\theta \in [0, \pi/2)}$ and define $\displaystyle{u_{\theta}=-\sin (\theta) \cdot  \ee_n+\cos(\theta) \cdot \ee_{n+1}}.$
\begin{Theorem} \label{th:41}
Let $f:M\longrightarrow \mathbb{R}^{n+1}$ be a complete, connected, properly embedded translating soliton and consider $\textstyle{\mathcal{C}:=\{x\in \mathbb{R}^{n+1} \; : \; \langle x,\ee_{1}\rangle^2+\langle u_{\theta},x\rangle^2 \leq r^2\}},$ where $r>0.$ Assume that $M$ is $C^{1}$-asymptotic  to two half-hyperplanes outside $\mathcal{C}$. Then we have one, and only one, of these two possibilities:\
\begin{enumerate}[(a)]
\item Both half-hyperplanes are contained in the same hyperplane $\Pi$ parallel to $\ee_{n+1}$ and $M$ coincides with $\Pi$;
\item Both half-hyperplanes are included in different parallel hyperplanes and $M$ coincides with a tilted grim reaper cylinder.
\end{enumerate}
\end{Theorem}
We will divide the proof of the theorem in many lemmas. But, before we start with the proof, we have to introduce some notation that we will use throughout the whole section.

Following \cite{MSHS15} and \cite{MPGSHS15}, let us define the foliation of $\mathbb{R}^{n+1}$ given by
\begin{equation}\label{3}
    \Pi(t)=\left\{x\in \mathbb{R}^{n+1}\; : \;\langle x,\ee_{1}\rangle=t\right\}.
\end{equation}
Furthermore, given $A \subset \mathbb{R}^{n+1}$ and $t \in \mathbb{R}$, let us define the sets
\begin{equation*}
 A_{+}(t)=\left\{x\in A\; : \;\langle x,\ee_{1}\rangle\geq t\right\}, A_{-}(t)=\left\{x\in A\; : \;\langle x,\ee_{1}\rangle\leq t\right\}
\end{equation*}   
\begin{equation*}
A^{+}(t)=\left\{x\in A\; : \;\langle x,u_{\theta}\rangle\geq t\right\}, A^{-}(t)=\left\{x\in A\; : \;\langle x,u_{\theta}\rangle\leq t\right\}.
\end{equation*}

Recall that we are assuming that the translating  velocity is $\ee_{n+1}$. From Lemma \ref{Characterization of the Hyperplane}, we work only in the case when the half-hyperplanes $\mathcal{H}_{1}$ and $\mathcal{H}_{2}$ lie in different and parallel hyperplanes to $\ee_{n+1}$. So, up to a translation, we can assume that the half-hyperplanes are contained in $\Pi\left(-\delta\right)$ and $\Pi\left(\delta\right)$, for a certain $\delta>0$. 
Let us begin by proving that both half-hyperplanes are parallel to $u_{\theta}$. 

\begin{Lemma}
The two connected components of $M$ which lie outside the cylinder $\mathcal{C}$ point in the same direction of $u_{\theta}.$
\end{Lemma}
\begin{proof}First of all, notice that  $M$ cannot be asymptotic to the half-hyperplanes
\begin{equation*}
\mathcal{H}_{1}=\left\{x\in \mathbb{R}^{n+1}\; : \; \langle x, u_{\theta}\rangle<r_{1}<0,\ x_{1}=-\delta\right\}
\end{equation*}
and
\begin{equation*}
\mathcal{H}_{2}=\left\{x\in \mathbb{R}^{n+1}\; : \; \langle x, u_{\theta}\rangle<r_{2}<0,\ x_{1}=\delta\right\}.
\end{equation*}
This is a trivial consequence of Theorem \ref{Strong Maximum Principle}, when one compares $M$ with a suitable copy of a tilted grim reaper transverse to the hyperplane $\Pi(0)$ (as we did at the end of the proof of Lemma \ref{Compactness Lemma}).

For the remaining cases, we proceed again by contradiction. Suppose at first that 
\begin{equation*}
\mathcal{H}_{1}=\left\{x\in \mathbb{R}^{n+1}\; : \; \langle x, u_{\theta}\rangle>r_{1}>0,\ x_{1}=-\delta\right\}
\end{equation*}
and
\begin{equation*}
\mathcal{H}_{2}=\left\{x\in \mathbb{R}^{n+1}\; : \; \langle x, u_{\theta}\rangle<r_{2}<0,\ x_{1}=\delta\right\}
\end{equation*}
for some $r_{1}>0$ and $r_{2}<0.$ Given $t$ and $l$ in $\R$, let $\mathcal{G}^{t,l}$ be the tilted grim reaper cylinder defined by
\begin{equation}\label{4}
\displaystyle{\mathcal{G}^{t,l}:=\left\{ F_\theta(x_1-t,x_2, \ldots,x_n)+t \ee_1+lu_{\theta} \; : \;
|x_{1}-t|<\tfrac{\pi}{2\cos(\theta)} ,\  (x_{2},\ldots,x_{n})\in \mathbb{R}^{n-1}
\right\}}
\end{equation}
\begin{figure}[htpb]
\begin{center}
\includegraphics[width=.60\textwidth]{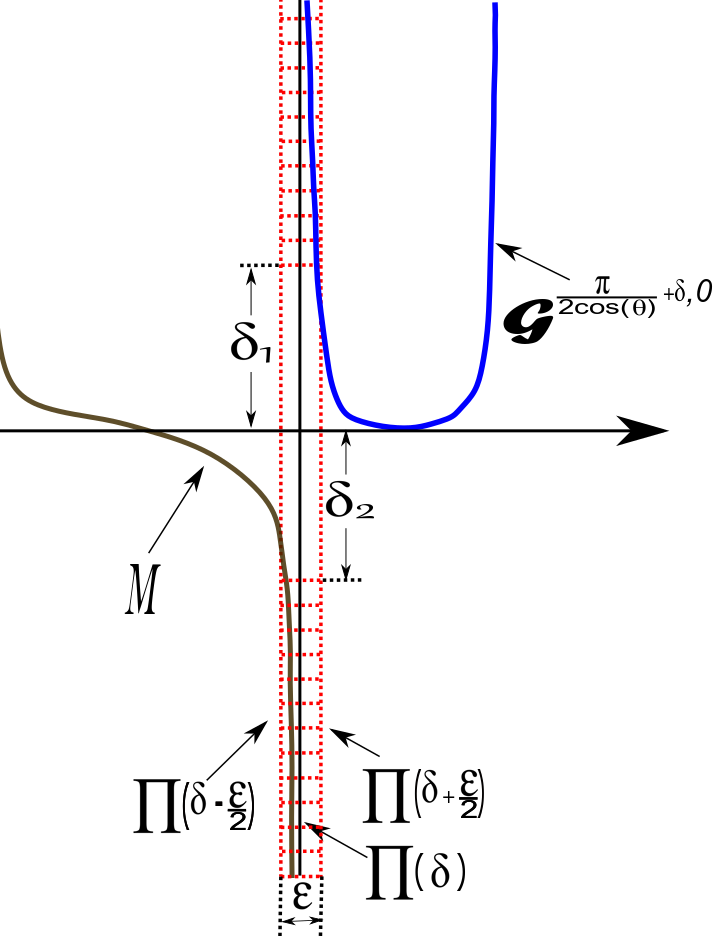}
\end{center}
\caption{Transversal section of the $M^{+}_{1}(\delta)$ and $\mathcal{G}^{\tfrac{\pi}{2\cos(\theta)}+\delta,0}$.} \label{figure 3}
\end{figure}
\noindent Consider $\mathcal{G}^{\tfrac{\pi}{2\cos(\theta)}+\delta,0}$, which lie in $\left(\delta,\delta+\tfrac{\pi}{\cos(\theta)}\right)\times\mathbb{R}^{n}$ (see Figure \ref{figure 3}). Note that it is asymptotic to the half-hyperplanes $\Pi\left(\delta\right)$ and $\Pi\left(\delta+\tfrac{\pi}{\cos(\theta)}\right).$ Fix $\epsilon \in (0,2\delta)$. Using  the fact that $\mathcal{G}^{\tfrac{\pi}{2\cos(\theta)}+\delta,0}$ is asymptotic to the half-hyperplanes outside the cylinder, then there exists $\delta_{1}>r_{1}$, depending only on $\epsilon$, such that\footnote{Here we are using the same notation of Lemma \ref{Compactness Lemma}.}
\begin{equation}\label{5}
\mathcal{G}^{\tfrac{\pi}{2\cos(\theta)}+\delta,0}\bigcap \mathcal{Z}^{+}_{\delta_{1}}\subset \left[\left(\delta,\delta+\tfrac{\epsilon}{2}\right)\times \mathbb{R}^{n}\right]\cap\{ x \in \R^{n+1} \; : \; \langle x, u_\theta\rangle >\delta_{1} \} .
\end{equation}
On the other hand, taking into account the asymptotic behaviour of $M$ and our assumptions about the wings, there exists a $\delta_{2}>-r_{2}$, depending only on $\epsilon$, such that
\begin{equation}\label{6}
M\bigcap \mathcal{Z}^{+}_{\delta_{2}}\subset \left[ \left(\delta-\tfrac{\epsilon}{2},\delta+\tfrac{\epsilon}{2} \right)\times \mathbb{R}^{n} \right]\cap\{ x \in \R^{n+1} \; : \; \langle x, u_\theta\rangle <\delta_{2} \}.
\end{equation}
From \eqref{5} and \eqref{6}, there exists a $t>0$ such that the tilted grim reaper cylinder $ \mathcal{G}^{\tfrac{\pi}{2\cos(\theta)}+\delta+t,-\delta_{1}-\delta_{2}-1}$ satisfies
\begin{equation*}
 \mathcal{G}^{\tfrac{\pi}{2\cos(\theta)}+\delta+t,-\delta_{1}-\delta_{2}-1}\bigcap M=\varnothing
\end{equation*}
Now, since $\epsilon \in (0, 2\delta)$, there is a finite $t_{0}$ such that either $M$ and $\mathcal{G}^{\tfrac{\pi}{2\cos(\theta)}+\delta+t_{0},-\delta_{1}-\delta_{2}-1}$ have a first point of contact or there is a sequence $\left\{p_{i}=(p^{1}_{i},\ldots,p^{n+1}_{i})\right\}$ in $M$ satisfying the next conditions: 
\begin{itemize}
\item[i.]$\{\langle p_{i},u_{\theta}\rangle\}$ is a bounded sequence;
\item[ii.] $\left\{(0,p^{2}_{i},\ldots,p^{n+1}_{i})-\langle p_{i},u_{\theta}\rangle u_{\theta}\right\}$ is an unbounded sequence;
\item[iii.]\begin{equation}\label{7}
    \lim_{i}\left\{\dist\left(p_{i},\mathcal{G}^{\tfrac{\pi}{2\cos(\theta)}+\delta+t_{0},-\delta_{1}-\delta_{2}-1}\right)\right\}=0,
\end{equation}
\end{itemize}
Notice that in this last case the sequence $\left\{p^{1}_{i}\right\}$ is bounded because to the asymptotic behaviour of $M$. Thus we can suppose $\{p^{1}_{i}\}\to p_{\infty}^{1}$ and $\{\langle p_{i},u_{\theta}\rangle\}\to p_{\infty}^{u_{\theta}}.$ In particular, from (\ref{7}), we have \[\textstyle{p_{\infty}^{1}\ee_{1}+p_{\infty}^{u_{\theta}}u_{\theta} \in \mathcal{G}^{\tfrac{\pi}{2\cos(\theta)}+\delta+t_{0},-\delta_{1}-\delta_{2}-1}}.\] According Theorem \ref{Strong Maximum Principle} and the asymptotic behaviour of $M$ the first case cannot happen. Regarding the second case, let us define the sequence \[\left\{M_{i}:=M-(0,p^{2}_{i},\ldots,p^{n+1}_{i})+\langle p_{i},u_{\theta}\rangle u_{\theta}\right\}.\] By Lemma \ref{Compactness Lemma}, up to a subsequence, we have that $M_{i}\rightharpoonup M_{\infty},$ where $M_{\infty}$ is a connected stationary integral varifold. Arguing as in Lemma \ref{Characterization of the Hyperplane}, we have 
\[p_{\infty}^{1}\ee_{1}+p_{\infty}^{u_{\theta}}u_{\theta} \in \spt M_{\infty}\cap \mathcal{G}^{\tfrac{\pi}{2\cos(\theta)}+\delta+t,-\delta_{1}-\delta_{2}-1}.\]
Thus, by Theorem \ref{Strong Maximum Principle} we get \[\textstyle{M_{\infty}=\mathcal{G}^{\tfrac{\pi}{2\cos(\theta)}+\delta+t_{0},-\delta_{1}-\delta_{2}-1}}.\] But this is impossible by the asymptotic behaviour of $M$. 

The case when
\begin{equation*}
\mathcal{H}_{1}=\left\{x\in \mathbb{R}^{n+1}\; : \;\langle x,u_{\theta}\rangle<r_{1}<0,\ x_{1}=-\delta\right\}
\end{equation*}
and
\begin{equation*}
\mathcal{H}_{2}=\left\{x\in \mathbb{R}^{n+1}\; : \; \langle x,u_{\theta}\rangle>r_{2}>0,\ x_{1}=\delta\right\}
\end{equation*}
can be excluded using a symmetric argument. This concludes the proof. 
\end{proof}

We shall see now that our embedded soliton must lie in the slab limited by the hyperplanes $\Pi(-\delta)$ and $\Pi(\delta)$.

\begin{Lemma}\label{slablimitation}
$M$ lies inside the slab $S:=(-\delta, \delta)\times \mathbb{R}^{n}.$
\end{Lemma}
\begin{proof}Assume that $\lambda:=\sup_{M}x_{1}>\delta$ and consider the domain
\begin{equation*}
    \Sigma=\left\{x\in M\; : \;\langle x,\ee_{1}\rangle\geq\tfrac{\delta+\lambda}{2}\right\}.
\end{equation*}
\begin{figure}[htpb]
\begin{center}
\includegraphics[width=.30\textwidth]{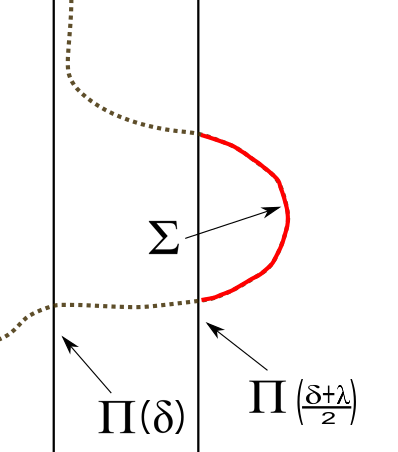} 
\end{center}
\caption{Transversal section of the $\Sigma$.} \label{figure 4}
\end{figure}

The asymptotic behaviour of $M$ tells us that the function $x\mapsto\langle x,u_{\theta}\rangle$ is bounded in $\Sigma$ (see Figure \ref{figure 4}). Now, by Lemma 3.3 we have
\begin{equation}\label{Sigma}
    \lambda=\sup_{\Sigma}x_{1}=\sup_{\partial\Sigma}x_{1}
\end{equation}
Therefore, because $\partial\Sigma\subset\Pi\left(\tfrac{\delta+\lambda}{2}\right)$ it follows that $ x_{1}(p)=\tfrac{\delta+\lambda}{2}$
for all $p\in \Sigma,$ which is absurd. Moreover, if $x_{1}(p)=\delta$ for any $p \in M,$ then $M$ and the hyperplane $\Pi(\delta)$ have  a point of contact. Hence by Theorem \ref{Strong Maximum Principle} we must have $M=\Pi(\delta)$ and again we arrive to a contradiction. This finishes the proof that $x_{1}(p)<\delta$ for all $p \in M.$ Using the same idea we shall obtain that $x_{1}(p)>-\delta$ for all $p \in M.$
\end{proof}

Now, we are going to show that the distance between the two half-hyperplanes is precisely $\frac{\pi}{\cos(\theta)},$ like in the tilted grim reaper cylinder. 

\begin{Lemma} 
We have $2\delta=\frac{\pi}{\cos(\theta)}.$
\end{Lemma}
\begin{proof} We argue again by contradiction. Assume at first that $2\delta>\tfrac{\pi}{\cos(\theta)}.$ By the asymptotic behaviour of $M$ we can place a tilted grim reaper cylinder $\mathcal{G}^{0,l}$ inside $S$, for sufficiently large $l$, so that
$\mathcal{G}^{0,l}\bigcap M=\varnothing $ (see Figure \ref{figure 5}).
\begin{figure}[htpb]
\begin{center}
\includegraphics[width=.45\textwidth]{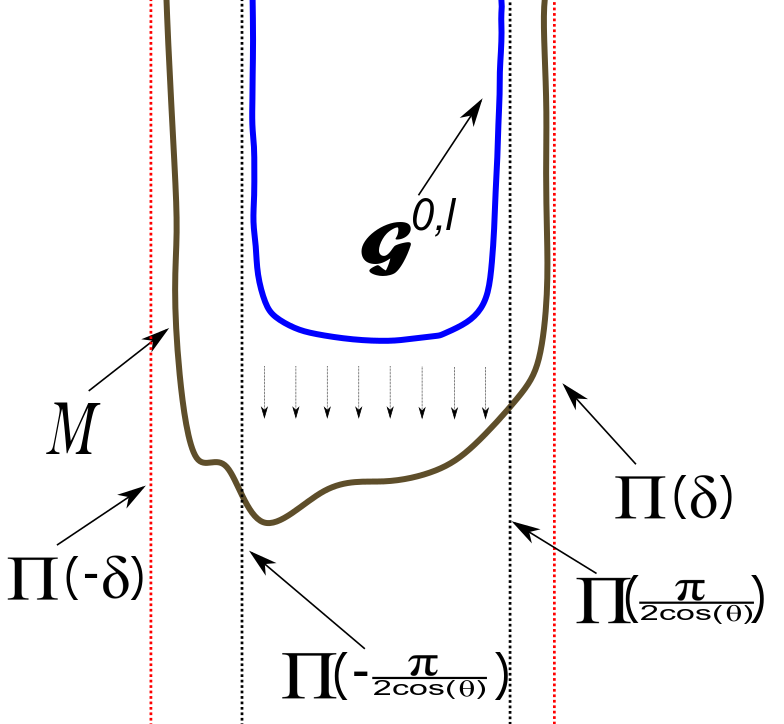} 
\end{center}
\caption{Transversal section of the behaviour of $\mathcal{G}^{0,l}$ with respect to $M$.} \label{figure 5}
\end{figure}

Now, consider $\mathcal{A}:=\left\{l\in \mathbb{R}\; : \; \mathcal{G}^{0,l}\bigcap M=\varnothing \right\}
$ and let $l_{0}=\inf\mathcal{A}$. Note that $l_{0}>-\infty$ by the asymptotic behaviour of $M.$ If $l_{0}\notin\mathcal{A}$, then $M$ and $\mathcal{G}^{0,l_{0}}$ have a point of contact. So $M=\mathcal{G}^{0,l_{0}}$ by Theorem \ref{Strong Maximum Principle}. But this is impossible. On the other hand, if $l_{0}\in \mathcal{A}$ then $\dist\left(M,\mathcal{G}^{0,l_{0}}\right)=0$. This means that there exists a sequence $\left\{p_{i}=(p_{i}^{1},\ldots,p_{i}^{n+1})\right\}$ in $M$ such that the sequences $\left\{p_{i}^{1}\right\}$ and $\left\{\langle p_{i},u_{\theta}\rangle\right\}$ are bounded, the sequence $\left\{(0,p_{i}^{2},\ldots,p_{i}^{n+1})-\langle p_{i},u_{\theta}\rangle u_{\theta}\right\}$ is unbounded and
$\dist\left(p_{i},\mathcal{G}^{0,l_{0}}\right)=0.$
Thus, up to a subsequence, one has  $p_{i}^{1}\to p_{\infty}^{1}$ and $\langle p_{i},u_{\theta}\rangle\to p_{\infty}^{u_{\theta}}$. Then, we consider the sequence of hypersurfaces $\left\{M_{i}\right\},$ where \[\textstyle{M_{i}:=M-(0,p_{i}^{2},\ldots,p_{i}^{n+1})+\langle p_{i},u_{\theta}\rangle u_{\theta}}.\] Using  Lemma \ref{Compactness Lemma} we can suppose that $M_{i}\rightharpoonup M_{\infty}$, where $M_{\infty}$ is a connected stationary integral varifold with $\displaystyle{p_{\infty}^{1}\ee_{1}+p_{\infty}^{u_{\theta}}u_{\theta} \in \spt M_{\infty}}.$ Hence $p_{\infty}^{1}\ee_{1}+p_{\infty}^{u_{\theta}}u_{\theta}$ is a point of contact between $\spt M_{\infty}$ and $\mathcal{G}^{0,l_{0}}$. Thus again by Theorem \ref{Strong Maximum Principle} we get that $\mathcal{G}^{0,l_{0}}=M,$ which contradicts our assumptions about the behaviour of $M$. Consequently $2\delta\leq\frac{\pi}{\cos(\theta)}.$ Comparing $M$ with a tilted grim reaper cylinder ``outside" $M$ we conclude $2\delta=\frac{\pi}{\cos(\theta)}.$ This completes the proof.
\end{proof}

In the next Lemma we prove that the connected components of $M \setminus\mathcal{C}$, that we will call from now on {\em the wings of $M$}, are graphs.

\begin{Lemma} \label{me}
If $t>0$ is sufficiently large, then the two connected components of $M^{+}(t)$ are graphs over the hyperplane $[\ee_{n+1}]^{\bot}.$ 
\end{Lemma}
\begin{proof} 
Observe that, if we take a sufficiently large $t$, then \[\displaystyle{M^{+}(t)\subset M_{+}\left(\frac{\pi}{2\cos(\theta)}-\tau\right)\cup M_{-}\left(-\frac{\pi}{2\cos(\theta)}+\tau\right)},\] for a small enough $\tau>0$. Therefore, we only need to prove that if $\delta$ is small enough, then  $\displaystyle{M_{+}\left(\frac{\pi}{2\cos(\theta)}-\tau\right)}$ is a graph over a subset of $[\ee_{n+1}]^{\bot}$. The case of $M_{-}\left(-\frac{\pi}{2\cos(\theta)}+\tau\right)$ is treated in a similar way. 

Fix a sufficiently small $\epsilon>0$, with $\epsilon<\frac{1}{8}$. Since $\mathcal{G}\left(=\mathcal{G}^{0,0}\right)$ and $M\setminus \mathcal{C}$ are $C^{1}$-asymptotic to the same half-hyperplane contained in  $\Pi\left(\frac{\pi}{2\cos(\theta)}\right),$  we can represent  $M_{+}\left(\frac{\pi}{2\cos(\theta)}-\tau\right)$ as graph over $\mathcal{G}.$  Hence, we can find a smooth map $\varphi:T_{\tau}:=\left(\tfrac{\pi}{2\cos(\theta)}-\tau,\tfrac{\pi}{2\cos(\theta)}\right)\longrightarrow \mathbb{R}$ such that 
\begin{equation}
    \sup_{T_{\tau}}|\varphi|<\epsilon\ \operatorname{and}\   \sup_{T_{\tau}}|D\varphi|<\epsilon.
\end{equation}
Moreover, the map $\widetilde{F}:T_{\tau}\times\mathbb{R}^{n-1}\longrightarrow \mathbb{R}^{n+1}$ given by
\begin{equation}\label{parametrization}
       \widetilde{F}=F_{\theta}+\varphi\nu_{F},
\end{equation}
is a parametrization of $M_{+}\left(\tfrac{\pi}{2\cos(\theta)}-\tau\right)$, where $F_{\theta}$ is the parametrization given by \eqref{tiltedgrim} and \[\displaystyle{\nu_{F_{\theta}}(x_{1},\ldots,x_{n})=\sin(x_{1}\cos(\theta))\ee_{1}-\cos(x_{1}\cos(\theta))u_{\theta}}.\]
Let us define the map $\textstyle{\Pi:\mathbb{R}^{n+1}\longrightarrow \mathbb{R}^{n}}$ given by  $\Pi(x_{1}\ldots,x_{n+1})=(x_{1}\ldots,x_{n})$
and consider its restriction
\begin{equation}
  \widetilde{\Pi}:=\Pi_{\left|\operatorname{int}\left(M_{+}\left(\tfrac{\pi}{2\cos(\theta)}-\tau\right)\right)\right.}:\operatorname{int}\left(M_{+}\left(\tfrac{\pi}{2\cos(\theta)}-\tau\right)\right) \longrightarrow T_{\tau}.
\end{equation}

Note that the image of $\widetilde{\Pi}$ lies on $T_\tau$, because for all $x\in \operatorname{int}\left(M_{+}\left(\tfrac{\pi}{2\cos(\theta)}-\tau\right)\right)$ we have \[\displaystyle{\tfrac{\pi}{2\cos(\theta)}-\tau<\langle x,\ee_{1} \rangle <\tfrac{\pi}{2\cos(\theta)}} ,\] by  the definition of $M_{+}\left(\tfrac{\pi}{2\cos(\theta)}-\tau\right).$ The idea here consists of  showing that $\widetilde{\Pi}$ is a diffeomorphism. To deduce this, by a standard topological argument, we only must check that:
\begin{enumerate}
  \item $\widetilde{\Pi}$ is a proper covering map;
  \item $\operatorname{int}\left(M_{+}\left(\tfrac{\pi}{2\cos(\theta)}-\tau\right)\right)$ is path connected.
\end{enumerate}

First, let us show that $\widetilde{\Pi}$ is a local diffeomorphism. Equivalently, let us show that $H>0$ on $M_{+}\left(\tfrac{\pi}{2\cos(\theta)}-\tau\right)$. Here we will use the parametrization \eqref{parametrization}. To do this, we use the orthonormal frame $\{E_{i}\}$ given by
\begin{eqnarray}\label{Part1}
\{E_{1}:=\cos(x_{1}\cos(\theta))\ee_{1}+\sin(x_{1}\cos(\theta))u_{\theta}\;, E_{j}:=\ee_{j},\ j\in\{2,\ldots,n-1\}\\
E_{n}:=\cos(\theta)\ee_{n}+\sin(\theta)\ee_{n+1} \;, \nu_{F}\}\nonumber
\end{eqnarray}
Representing the vectors with respect this basis, we obtain that the unit normal $N_{\widetilde{F}}$ of $\widetilde{F}$ is given by the formula
\begin{equation}\label{normal}
\Omega\cdot N_{\widetilde{F}}=AE_{1}+B\textstyle\sum_{j=2}^{n-1}(-1)^{j\left[n-\frac{j+1}{2}\right]}\partial_{x_{j}}\varphi E_{j}+C\partial_{x_{n}}\varphi E_{n}+B\nu_{F},
\end{equation}
where
\begin{equation}\label{A}
A:=(-1)^{n-2}\left(\sin(\theta)\sin(x_{1}\cos(\theta))\partial_{x_{n}}\varphi-\cos(x_{1}\cos(\theta))\partial_{x_{1}}\varphi\right)
\end{equation}
\begin{equation}\label{B}
B:=1+\varphi\cos(\theta)\cos(x_{1}\cos(\theta))
\end{equation}
\begin{equation}\label{C}
C:=(-1)^{n-1}\cos(\theta)\left(1+\varphi\cos(\theta)\cos(x_{1}\cos(\theta))\right)
\end{equation}
\begin{eqnarray}\label{D}
\Omega^2 &:=&\left[\sin(\theta)\sin(x_{1}\cos(\theta))\partial_{x_{n}}\varphi-\cos(x_{1}\cos(\theta))\partial_{x_{1}}\varphi\right]^{2}\\
&+&[1+\varphi\cos(\theta)\cos(x_{1}\cos(\theta))]^2\left(1+\textstyle\sum_{j=2}^{n}\left(\partial_{x_{j}}\varphi\right)^2\right)\nonumber\\
&+&\cos^2(\theta)\left[1+\varphi\cos(\theta)\cos(x_{1}\cos(\theta))\right]^{2}\left(\partial_{x_{n}}\varphi\right)^2\nonumber
\end{eqnarray}
In particular by (\ref{TS}), we have
\begin{eqnarray}\label{curvature}
\nonumber \tfrac{\Omega\cdot H}{\cos{\theta}\cos(x_{1}\cos(\theta))}&:=&1+\varphi\cos(\theta)\cos(x_{1}\cos(\theta))+(-1)^{n}\sin(x_{1}\cos(\theta))\partial_{x_{1}}\varphi\\
&+&(-1)^{n}\sin(\theta)\cos(x_{1}\cos(\theta))\partial_{x_{n}}\varphi+(-1)^{n}\cos(\theta)\sin(\theta)\varphi\partial_{x_{n}}\varphi
\end{eqnarray}
Now by our assumptions about $\epsilon$, $\varphi$ and $D\varphi$ we immediately deduce that $H>0$ at all $p \in M_{+}\left(\tfrac{\pi}{2\cos(\theta)}-\tau\right).$ Hence, $\widetilde{\Pi}$ is a local diffeomorphism.

The previous argument also implies that $\widetilde \Pi$ is onto. Indeed, if it does not there would be a vertical cylinder which intersects $T_\tau$ but it would not intersect the set $M_{+}\left(\tfrac{\pi}{2\cos(\theta)}-\tau\right)$. Taking into account the asymptotic behaviour of $M$, we could translate horizontally this cylinder until having a first contact with
$$\operatorname{int} \displaystyle{\left(M_{+}\left(\tfrac{\pi}{2\cos(\theta)}-\tau\right) \right)}.$$ At this first contact the normal vector field to $M$ would be horizontal, which is absurd because we have proved that $H>0$ on $M_{+}\left(\tfrac{\pi}{2\cos(\theta)}-\tau\right).$

Finally, let us check that $\widetilde{\Pi}$ is proper. Let $K \subset T_{\tau}$ a compact set and $\left\{p_{i}\right\}_{i\in\mathbb{N}}$ be a sequence on $\widetilde{\Pi}^{-1}(K).$ Note that the sequence $\left\{p_{i}\right\}_{i\in \mathbb{N}}$ is bounded, because of the asymptotic behaviour of $M$ and the fact that $\dist\left(K,\partial T_{\tau}\right)>0$.  So, up to a subsequence, we can assume that $p_{i}\to p_{\infty}$. Since the set $\widetilde{\Pi}^{-1}(K)$ is closed, it follows that $p_{\infty}\in \widetilde{\Pi}^{-1}(K).$ This proves that $\widetilde{\Pi}^{-1}(K)$ is compact. 

At this point, we have that any connected component of $\operatorname{int} \left(M_{+}\left(\tfrac{\pi}{2\cos(\theta)}-\tau\right) \right)$ is a graph over $T_\tau$. But only one of them contains the wing. This means that if there were another connected component, $\Sigma$, then the function $x\mapsto\langle x,u_{\theta}\rangle$ would be bounded on $\Sigma$ and \(\displaystyle{\partial \Sigma \subset \Pi\left(\tfrac{\pi}{2\cos(\theta)}-\tau\right)},\) which is contrary to Lemma \ref{Maximum Principle for Unbounded Domain}. Repeating the same argument we should obtain that $M_{-}\left(-\tfrac{\pi}{2\cos(\theta)}+\tau\right)$ is graph over the hyperplane $[\ee_{n+1}]^{\perp}.$
\end{proof}

Now we are going to show that is possible to place a tilted grim reaper cylinder below $M$. This means that $M$ lies in the convex region limited by the tilted grim reaper cylinder. Without loss of generality we can assume from now on that $\inf_{M}\langle x,u_{\theta}\rangle=0.$ 

\begin{Lemma}  \label{lem:inside}
There is a tilted grim reaper cylinder that contains $M$ ``inside" it, i.e., $M$ lies in the convex region of the complement of a tilted grim reaper cylinder.
\end{Lemma}
\begin{proof}Consider the family of ``half"-tilted grim reaper cylinders
\begin{equation}
\mathcal{G}^{t,-\epsilon}_{\pm}:=\left\{x \in \mathcal{G}^{0,-\epsilon}\; : \;\pm\langle x,\ee_{1}\rangle\geq 0\right\}\pm t\ee_{1}
\end{equation}
where $\epsilon>0$ is fixed and $t \in [0,\infty).$
\begin{figure}[htpb]
\begin{center}
\includegraphics[width=.50\textwidth]{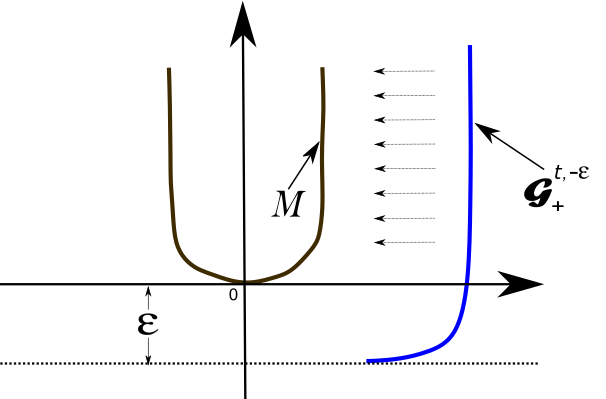} 
\end{center}
\caption{Transversal section of the behaviour of $\mathcal{G}^{t,-\epsilon}_{+}$ with respect to $M$.} \label{figure 6}
\end{figure}

Let us work with the ``half''-tilted grim reaper cylinder $\mathcal{G}^{t,-\epsilon}_{+}.$ By taking a sufficiently large $t_{0}$, we obtain $\mathcal{G}^{t_{0},-\epsilon}_{+}\cap M=\varnothing$. Hence the set $\mathcal{A}$ defined by $$\textstyle{\mathcal{A}:=\{t\in [0,\infty)\; : \; \mathcal{G}^{t,-\epsilon}_{+}\cap M=\varnothing\}}$$ is not empty. Take $s_{0}=\inf\mathcal{A}.$  We claim that $s_{0}=0.$  Otherwise, we have two possibilities for $s_{0}>0$: either $s_{0}\in \mathcal{A}$ or $s_{0}\notin \mathcal{A}$. If $s_{0}\notin\mathcal{A}$ then $\mathcal{G}^{t,-\epsilon}_{+}\cap M\neq\varnothing$ and since $\partial \mathcal{G}^{t,-\epsilon}_{+}\cap M=\varnothing,$  we conclude that $\mathcal{G}_{+}^{t,-\epsilon}\subset M$, by Theorem \ref{Strong Maximum Principle}, but this is absurd because \[0=\displaystyle{\inf_{M}\langle x,u_{\theta}\rangle>\inf_{\mathcal{G}^{t,-\epsilon}}\langle x,u_{\theta}\rangle=-\epsilon}.\] 

If $s_{0} \in \mathcal{A}$ then $\dist\left(\mathcal{G}^{t,-\epsilon}_{+}, M\right)=0$. This tell us that there exists a sequence $\left\{p_{i}=(p_{i}^{1},\ldots,p_{i}^{n+1})\right\}$ in $M$ such that:
\begin{enumerate}[i.]
\item$\textstyle\lim_{i}\dist\left(p_{i}, \mathcal{G}^{t,-\epsilon}_{+}\right)=0$;
\item$\{p_{i}^{1}\}\to p_{\infty}^{1}$ and $a<p_{i}^{1}-t<b$, where $0<a<b<\tfrac{\pi}{2\cos(\theta)}$ are constants;
\item$\{\langle p_{i},u_{\theta}\rangle\}\to p_{\infty}^{u_{\theta}}$;
\item The sequence $\left\{(0,p_{i}^{2},\ldots,p_{i}^{n+1})-\langle p_{i},u_{\theta}\rangle u_{\theta}\right\}$ is unbounded.
\end{enumerate}
In this case, consider the sequence of hypersurfaces 
$$\left\{M_{i}=M-(0,p_{i}^{2},\ldots,p_{i}^{n+1})+\langle p_{i},u_{\theta}\rangle u_{\theta}\right\}. 
$$
By Lemma \ref{Compactness Lemma} we can suppose $M_{i}\rightharpoonup M_{\infty}$, where $M_{\infty}$ is a connected stationary integral varifold. Since $\textstyle{p_{\infty}^{1}\ee_{1}+ p_{\infty}^{\theta}u_{\theta}\in \mathcal{G}^{t,-\epsilon}_{+}\cap \spt M_{\infty}},$ using again Theorem \ref{Strong Maximum Principle} we get $\mathcal{G}^{t,-\epsilon}_{+}=M_{\infty}.$ But this again contradicts the asymptotic behaviour of $M$. Therefore $\inf \mathcal{A}=0$, and 
\begin{equation*}
\mathcal{G}^{0,-\epsilon}_{+}\cap M=\varnothing, 
\end{equation*}
for all $\epsilon>0.$ A symmetric argument shows that 
$
\mathcal{G}^{0,-\epsilon}_{-}\cap M=\varnothing.
$
Thus
$
\mathcal{G}^{0,-\epsilon}\cap M=\varnothing.
$

This completes the proof.
\end{proof}

As an application of the previous lemma, we will conclude that our hypersurface is in fact a graph over the hyperplane $[\ee_{n+1}]^{\perp}.$ 

\begin{Lemma}
$M$ is a graph over $\left(-\tfrac{\pi}{2 \cos \theta},\tfrac{\pi}{2 \cos \theta}\right) \times \R.$
\end{Lemma}
\begin{proof}For each $i \in\mathbb{N}$ consider the sets \[\displaystyle{T_{i}:=\left\{v\in \R^{n+1}\;: \langle v,E_{n}\rangle\geq i\right\}},\] where $E_{n}:=\cos(\theta)\ee_{n}+\sin(\theta)\ee_{n+1}$, and call $\alpha:=\displaystyle{\lim_{i}\inf_{T_{i}\cap M} \langle x,u_{\theta}\rangle}$. Consider a sequence $\left\{p_{i}=\left(p^{1}_{i},\ldots,p^{n+1}_{i}\right)\right\}$ in $M$ such that:
\begin{itemize}
\item[i.]$p_{i}\in T_{i}\cap M$ and $\langle p_{i},u_{\theta}\rangle-\textstyle{\inf_{T_{i}\cap M}\langle x,u_{\theta}\rangle<\tfrac{1}{i}}$
\item[ii.]$\{p_{i}^{1}\}\to p_{\infty}^{1}$ and $-\tfrac{\pi}{2\cos(\theta)}<p_{\infty}^{1}<\tfrac{\pi}{2\cos(\theta)}$;
\item[iii.]$\{\langle p_{i},u_{\theta}\rangle\}\to \alpha$;
\end{itemize}
Consider again the sequence of hypersurfaces $$\left\{M_{i}=M-(0,p_{i}^{2},\ldots,p_{i}^{n+1})+\langle p_{i},u_{\theta}\rangle u_{\theta}\right\}.$$ 

By Lemma \ref{Compactness Lemma}, we can suppose $M_{i}\rightharpoonup M_{\infty}$, where $M_{\infty}$ is a connected stationary integral varifold. Since $p_{\infty}^{1}\ee_{1}+\alpha u_{\theta}\in  \spt M_{\infty},$ it follows that $\textstyle{\inf_{\spt M_{\infty}}\langle x,u_{\theta}\rangle \leq\alpha}.$ We claim that $\alpha=\textstyle{\inf_{\spt M_{\infty}}\langle x,u_{\theta}\rangle}.$ Indeed, take any $p \in \R^{n+1}$ such that $\langle p,u_{\theta}\rangle<\alpha.$ Let $B_{r}(p)$ be the open ball in $\mathbb{R}^{n+1},$  where $r\in\left(0,\tfrac{\alpha-\langle p,u_{\theta}\rangle}{4}\right).$ Note that $B_{r}(p)\cap \Pi_{\alpha}=\varnothing,$ where $\Pi_{\alpha}=[u_{\theta}]^{\perp}+\alpha u_{\theta}.$ Take any $\epsilon\in \left(0,\tfrac{\alpha-\langle p,u_{\theta}\rangle)}{4}\right).$ By the definition of $\alpha,$ there is a $i_{0}$ such that if $i>i_{0}$ then 
\begin{equation}\label{equa}
\inf_{T_{i}\cap M}\langle x,u_{\theta}\rangle>\alpha-\epsilon>0.
\end{equation}

Let $\varphi:\mathbb{R}^{n+1}\longrightarrow[0,\infty)$ be a smooth function such that:
\begin{itemize}
\item[a.]$\spt \varphi \subset B_{r}(p)$
\item[b.]$\varphi_{|_{ B_{\frac{r}{2}}(p)}}\equiv1$
\end{itemize}
By (\ref{equa}), if $i>i_{0}$ we have
\begin{equation*}
\int_{\mathbb{R}^{n+1}}\varphi d\mu_{M_{i}}=0,
\end{equation*}
where $d\mu_{M_i}$ denotes the Radon (Riemannian) measure associated to $M_i$. This implies that $\textstyle{\int_{\mathbb{R}^{n+1}}\varphi d\mu_{M_{\infty}}=0}$ and so $p \notin \spt M_{\infty}$. Consequently $\alpha=\textstyle\inf_{\spt M_{\infty}}\langle x,u_{\theta}\rangle.$ Now, we claim that $M_{\infty}$ coincides with the tilted grim reaper cylinder. The proof follows by the same idea as in Lemma 4.5 to prove this (see Figure \ref{figure 6}). Consider the ``half"-tilted grim reaper cylinder
\begin{equation}
\mathcal{G}^{t,\alpha-\epsilon}_{+}:=\left\{x \in \mathcal{G}^{0,\alpha-\epsilon}\; : \;\langle x,\ee_{1}\rangle\geq 0\right\}+t\ee_{1}
\end{equation}
where $\epsilon>0$ and $t \in [0,\infty).$ Take a sufficiently large $t_{0}$ so that
\begin{equation*}
\mathcal{G}^{t_{0},\alpha-\epsilon}_{+}\cap \spt M_{\infty}=\varnothing.
\end{equation*}
By Lemma 4.5 this is possible. Consider the set $$\displaystyle{\mathcal{A}=\{t\in[0,+\infty)\; : \;\mathcal{G}^{t,\alpha-\epsilon}_{+}\cap \spt M_{\infty}=\varnothing\}},$$ which is non-empty.  We will show that $\inf\mathcal{A}=0.$ Indeed, otherwise, $s_{0}=\inf\mathcal{A}>0$ satisfy one of the following conditions: 
\begin{itemize}
\item[a.]$\mathcal{G}^{s_{0},\alpha-\epsilon}_{+}$ and $\spt M_{\infty}$ have a point of contact;
\item[b.]$\dist\left(\mathcal{G}^{s_{0},\alpha-\epsilon}_{+},\spt M_{\infty}\right)=0.$ 
\end{itemize}
According to Theorem \ref{Strong Maximum Principle} and Lemma 4.5, the first case is not possible. Regarding the second case, by Lemma 4.5 there exists a sequence $\left\{z_{i}=(z_{i}^{1},\ldots,z_{i}^{n+1})\right\}$ in $\spt M_{\infty}$ such that:
\begin{itemize}
\item[i.]$\displaystyle\lim_{i}\dist\left(z_{i}, \mathcal{G}^{s_{0},\alpha-\epsilon}_{+}\right)=0$;
\item[ii.]$\{z_{i}^{1}\}\to z_{\infty}^{1}$ and $a<z_{i}^{1}-t<b$ where $0<a<b<\frac{\pi}{2\cos(\theta)}$ are constants;
\item[iii.]$\{\langle z_{i},u_{\theta}\rangle\}\to z_{\infty}^{u_{\theta}}$;
\item[iv.]The sequence $\left\{(0,z_{i}^{2},\ldots,z_{i}^{n+1})-\langle z_{i},u_{\theta}\rangle u_{\theta}\right\}$ is unbounded;
\item[v.]$\Theta(\spt M_{\infty},z_{i})\geq 1,$ where $\Theta(\spt M_{\infty},\cdot)$ is the density of the Radon measure associated to $M_{\infty}.$
\end{itemize}
At this point, let us consider the sequence $$\displaystyle{\left\{\mathfrak{M}_{i}:=M_{\infty}-(0,z_{i}^{2},\ldots,z_{i}^{n+1})+\langle z_{i},u_{\theta}\rangle u_{\theta}\right\}}.$$ By Remark \ref{re:20} we can suppose $\mathfrak{M}_{i}\rightharpoonup \mathfrak{M}_{\infty},$ where $\mathfrak{M}_{\infty}$ is a connected stationary integral varifold. Following the same arguments which we used to show that $p_{\infty}\in \spt M_{\infty}$ in Lemma \ref{Characterization of the Hyperplane}, we see that \[\displaystyle{z_{\infty}^{1}\ee_{1}+z_{\infty}^{u_{\theta}}u_{\theta}\in \spt \mathfrak{M}_{\infty}\cap\  \mathcal{G}^{s_{0},\alpha-\epsilon}_{+}}.\] Moreover note that item ii implies that $z_{\infty}^{1}\ee_{1}+z_{\infty}^{u_{\theta}}u_{\theta}$ is an interior point of $\mathcal{G}^{s_{0},\alpha-\epsilon}_{+}.$ Therefore by Theorem \ref{Strong Maximum Principle} and Lemma 4.5 we obtain to a contraction. Thus, $\inf\mathcal{A}=0$ and \[\displaystyle{\mathcal{G}^{0,\alpha-\epsilon}_{+}\cap \spt M_{\infty}=\varnothing},\] because $\epsilon>0$ and $\textstyle{\inf_{\spt M_{\infty}}\langle x,u_{\theta}\rangle=\alpha}.$ Similarly, we deduce that \[\displaystyle{\mathcal{G}^{0,\alpha-\epsilon}_{-}\cap \spt M_{\infty}=\varnothing}.\] Hence 
$
\mathcal{G}^{0,\alpha-\epsilon}\cap \spt M_{\infty}=\varnothing.
$

Now, taking $\epsilon\to 0^{+}$ and using the fact that \[\displaystyle{\inf_{\spt M_{\infty}}\langle x,u_{\theta}\rangle=\min_{\spt M_{\infty}}\langle x,u_{\theta}\rangle=\alpha}\] we conclude that $\spt M_{\infty}$ touches the tilted grim reaper cylinder $\mathcal{G}^{0,\alpha}$ at $p_{\infty}^{1}\ee_{1}+\alpha u_{\theta}$. In particular, by Theorem \ref{Strong Maximum Principle} we conclude that $M_{\infty}=\mathcal{G}^{0,\alpha}$. This concludes the proof of our claim. Moreover, by Theorem \ref{Allard's Regularity Theorem} we have $M_{i}\longrightarrow M_{\infty}=\mathcal{G}^{0,\alpha},$ with multiplicity one. 

Let us consider the sets \[\displaystyle{S_{i}:=\left\{v\in \R^{n+1}\;: \langle v,E_{n}\rangle \leq-i\right\}},\] where $i \in\mathbb{N}$. Take $\displaystyle{\beta=\lim_{i}\inf_{S_{i}\cap M} \langle x,u_{\theta}\rangle}$. Let $\left\{q_{i}=\left(q^{1}_{i},\ldots,q^{n+1}_{i}\right)\right\}$ be  a sequence in $M$ such that:
\begin{itemize}
\item[i.]$q_{i}\in S_{i}\cap M$ and $\langle q_{i},u_{\theta}\rangle-\textstyle{\inf_{S_{i}\cap M}\langle x,u_{\theta}\rangle<\tfrac{1}{i}}$
\item[ii.]$\{q_{i}^{1}\}\to q_{\infty}^{1}$ and $-\tfrac{\pi}{2\cos(\theta)}<q_{\infty}^{1}<\tfrac{\pi}{2\cos(\theta)}$;
\item[iii.]$\{\langle q_{i},u_{\theta}\rangle\}\to \beta$.
\end{itemize}
Then, reasoning as before, we obtain \[\displaystyle{N_{i}:=M-(0,q_{i}^{2},\ldots,q_{i}^{n+1})+\langle q_{i},u_{\theta}\rangle u_{\theta}\longrightarrow\mathcal{G}^{0,\beta}},\] with multiplicity one.

By Lemma 4.4, we know that there exists a sufficiently large $t_{0}$, so that $M^{+}(t_{0})$ is a graph over an open set in the hyperplane $[\ee_{n+1}]^{\perp}.$ Moreover, we can choose at the same time a small enough $\tau>0$ so that $$\textstyle{M_{+}\left(\tfrac{\pi}{2\cos(\theta)}-2\tau\right)\cup M_{+}\left(-\tfrac{\pi}{2\cos(\theta)}+2\tau\right)\subset M^{+}(t_{0})}.$$ Hence, there is $i_{0} \in\n$ such that:

\begin{itemize}
\item[a.]There exist strictly increasing sequences of positive numbers $\{m^{1}_{i}\}$, $\{m^{2}_{i}\}$, $\{n^{1}_{i}\}$  and $\{n^{2}_{i}\}$ so that
\begin{equation*}
m^{1}_{i} <m^{2}_{i}\ \operatorname{and}\ -n^{1}_{i}<-n^{2}_{i}, \quad \mbox{for all $i>i_{0};$};
\end{equation*}
\item[b.]There exist smooth functions: 
\begin{equation}
\varphi_{i}:\left(-\tfrac{\pi}{2\cos(\theta)}+\tau,\tfrac{\pi}{2\cos(\theta)}-\tau\right)\times \left(m^{1}_{i} ,m^{2}_{i}\right)^{n-1}\longrightarrow\mathbb{R}
\end{equation}
and 
\begin{equation}
\phi_{i}:\left(-\tfrac{\pi}{2\cos(\theta)}+\tau,\tfrac{\pi}{2\cos(\theta)}-\tau\right)\times \left(-n^{1}_{i} ,-n^{2}_{i}\right)^{n-1}\longrightarrow\mathbb{R}
\end{equation}
satisfying
\begin{equation}
|\varphi_{i}|<\tfrac{1}{i}, |D\varphi_{i}|<\tfrac{1}{i}, |\phi_{i}|<\tfrac{1}{i}\ \mbox{and}\ |D\phi_{i}|<\tfrac{1}{i}, \; \mbox{for all $i>i_{0}$}
\end{equation}
and such that the hypersurfaces
\begin{eqnarray*}
R_{i}:=&\{(x_{1},\ldots,x_{n+1})\in M : -\tfrac{\pi}{2\cos(\theta)}+\tau<x_{1}<\tfrac{\pi}{2\cos(\theta)}-\tau\\
&(x_{2},\ldots,x_{n-1})\in \left(m^{1}_{i} ,m^{2}_{i}\right)^{n-2}, \langle x,E_{n}\rangle\in\left(m^{1}_{i} ,m^{2}_{i}\right)\}\nonumber
\end{eqnarray*}
and
\begin{eqnarray*}
L_{i}:=&\{(x_{1},\ldots,x_{n+1})\in M : -\tfrac{\pi}{2\cos(\theta)}+\tau<x_{1}<\tfrac{\pi}{2\cos(\theta)}-\tau\\
&(x_{2},\ldots,x_{n-1})\in \left(-n^{1}_{i} ,-n^{2}_{i}\right)^{n-2}, \langle x,E_{n}\rangle\in\left(-n^{1}_{i},-n^{2}_{i}\right)\}\nonumber
\end{eqnarray*}
can be written as graphs of functions $\varphi_{i}$ and $\phi_{i}$, respectively, over the corresponding pieces of the tilted grim reaper cylinder as in the proof of Lemma \ref{me}, where $E_{n}:=\cos(\theta)\ee_{n}+\sin(\theta)\ee_{n+1}$. 
\end{itemize} 
Now following the same idea as in Lemma \ref{me}, we see that $R_{i}$ and $L_{i}$ are graphs over domains in the hyperplane $[\ee_{n+1}]^{\perp}$ (for $i_{0}$ large enough.) Note that $R_{i}$ and $L_{i}$ are connected because they are graphs over the connected sets and the convergence has multiplicity one. 
\begin{figure}[htpb]
\begin{center}
\includegraphics[width=.35\textwidth]{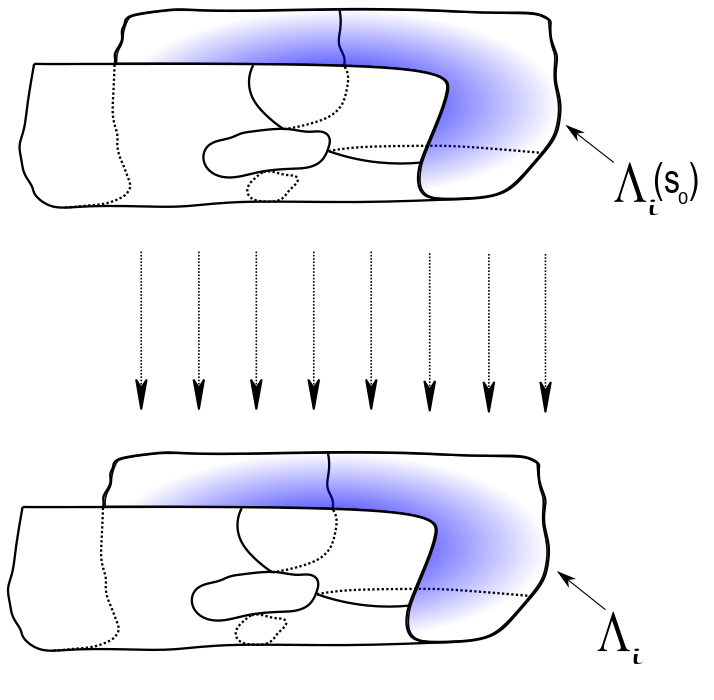}
\end{center}
\caption{Picture of $\Lambda_{i}$ and $\Lambda_{i}(s_{0})$.} \label{figure 8}
\end{figure}
Finally, let us consider the exhaustion $\left\{\Lambda_{i}\right\}$ of $M$ by compacts sets given by
\begin{eqnarray}\label{lambdai}
\Lambda_{i}:=\{x=(x_{1},\ldots,x_{n+1}) \in M \;:(x_{2},\ldots,x_{n-1})\in \left[-a_{i},b_{i}\right]^{n-2}\\
\langle x,E_{n}\rangle\in \left[-a_{i},b_{i}\right]\;:\langle x,u_{\theta}\rangle\leq i\}\nonumber
\end{eqnarray}

where $a_{i}=\frac{n_{i}^{1}+n_{i}^{2}}{2}$ and $b_{i}=\frac{m_{i}^{1}+m_{i}^{2}}{2}.$

Since $M^{+}(t_{0})$, $R_{i}$ and $L_{i}$ are vertical graphs, then a small strip $B_{i}$ around the boundary of $\Lambda_{i}$ is a graph over the hyperplane $[\ee_{n+1}]^{\perp}.$ Using the Rado's classical argument, the former fact implies that $\Lambda_i$ is a graph over the hyperplane $[\ee_{n+1}]^{\perp}$ if $i>i_{0}.$ 
Indeed, assume to the contrary that this is not true. Consider the family
\begin{equation*}
\left\{\Lambda_{i}(s):=\Lambda_{i}+s\ee_{n+1}\right\}_{s\in \mathbb{R}}
\end{equation*}
of translations of $\Lambda_{i}$ into the direction of $\ee_{n+1}.$ Since $\Lambda_{i}$ is compact there exists a sufficiently large $s_{0}$ so that
\begin{equation*}
\Lambda_{i}(s_{0})\cap\Lambda_{i}=\varnothing.
\end{equation*}
Now move $\Lambda_{i}(s_{0})$ into direction of $-\ee_{n+1}$(see Figure \ref{figure 8}). Since $\Lambda_{i}$ is not a graph and $B_{i}\cap\left\{B_{i}+s\ee_{n+1}\right\}=\varnothing,$ because $B_{i}$ is a graph over a subset of $[\ee_{n+1}]^{\bot}$. Then there exists a $s_{1}\in (0,s_{0})$ such that $\Lambda_{i}(s_{1})$ has a point of contact at interior with $\Lambda_{i}$. Therefore $\Lambda_{i}(s_{1})=\Lambda_{i}$, but this gives us to a contraction. Hence each $\Lambda_{i}$ must be a graph. Since $\textstyle{\bigcup_{i}\Lambda_{i}=M}$, then $M$ is also a vertical graph.
\end{proof}
Observe that the previous lemma implies that $H>0$. Hence, given any $v\in\R^{n+1}$ if $\xi_{v}=\langle\xi,v\rangle$, then $h_{v}=\frac{\xi_{v}}{H}$ are well defined on $M$, where $\xi$ stand for the Gauss map of $M$ and $H$ is the mean curvature of $M$. 

\begin{Lemma}
The functions $\xi_{v}$, $H$ and $h_{v}=\frac{\xi_{v}}{H}$ satisfy the following equalities
\begin{equation} \label{primera}
\Delta \xi_{v}+\langle\nabla\xi_{v},\nabla x_{n+1}\rangle+|A|^{2}\xi_{v}=0,
\end{equation}
\begin{equation} \label{segunda}
\Delta H+\langle\nabla H,\nabla x_{n+1}\rangle+|A|^{2}H=0,
\end{equation}
\begin{equation} \label{tercera}
\Delta h_{v}+\langle\nabla h_{v},\nabla (x_{n+1}+2\log H)\rangle=0.
\end{equation}

\end{Lemma}
\begin{proof} Equation \eqref{segunda} was proved in \cite[Lemma 2.1]{MSHS15}. Similarly, the proof of \eqref{primera} can be found in \cite[eq. (3.4.10)]{GARCIA16}. Equation \eqref{tercera} follows from the combination of \eqref{primera} and \eqref{segunda}.
\end{proof}

Now we will show that our hypersurface is $C^2-$asymptotic to two half-hyperplanes with respect to the Euclidean metric.
\begin{Lemma} \label{lem:grimreaper}
The hypersurface $M$ is $C^{2}-$asymptotic outside the cylinder to two half-hyperplanes with respect to the Euclidean metric.
\end{Lemma}
\begin{proof} To prove this lemma, we will need the following claim.
\begin{claim} There exist a tilted grim reaper cylinder inside the region that lie above $M$.
\end{claim}
\noindent{\it Proof of claim:} Note that if $t$ is sufficiently large then $M^{+}(t)$ is graph over the hyperplane $\Pi(0).$  Now take the tilted grim reaper $\mathcal{G}^{0,t}$.  We will show that it lies in the region above $M$. As in Lemma \ref{lem:inside}, let us consider the family
\begin{equation*}
\left\{M_{*}(s):=M_{+}(0)+s\ee_{1}\right\}_{s \in [0,+\infty)}.
\end{equation*}

Taking into account the asymptotic behaviour of $M$, there exists a sufficiently large $s_{0}>0$ so that $M_{*}(s_{0})\cap \mathcal{G}^{0,t}=\varnothing.$ Applying the same argument as in Lemma \ref{lem:inside} and the fact that $M^{+}(t)$ is graph over $\Pi(0)$, we conclude that $\inf\mathcal{A}=0$, where $$\mathcal{A}:=\left\{s\in [0,+\infty)\; : \;M_{*}(s)\cap\mathcal{G}^{0,t}=\varnothing\right\}.$$ 
In particular, it holds that $M_{+}(0)\cap \mathcal{G}^{0,t}=\varnothing.$  Applying the same argument to the family 
\begin{equation*}
\left\{M^{*}(s):=M_{-}(0)-s\ee_{1}\right\}_{s \in [0,+\infty)},
\end{equation*}
we obtain $M_{-}(0)\cap \mathcal{G}^{0,t}=\varnothing$. Hence $M\cap\mathcal{G}^{0,t}=\varnothing $ and this proves the claim (see Figure \ref{figure 9}).

\begin{figure}[htpb]
\begin{center}
\includegraphics[height=.33\textheight]{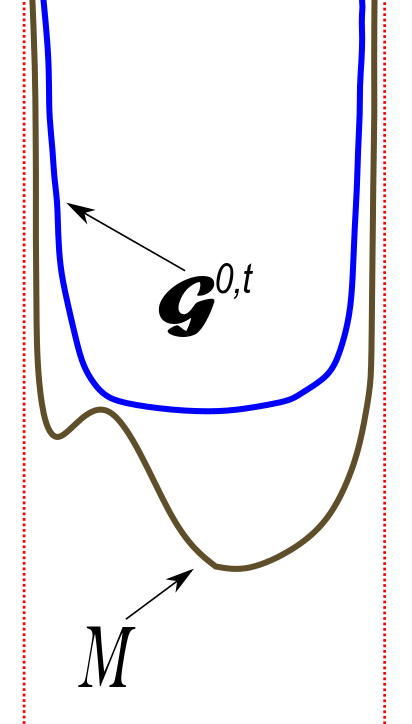}\hspace{.5cm}
\includegraphics[height=.33\textheight]{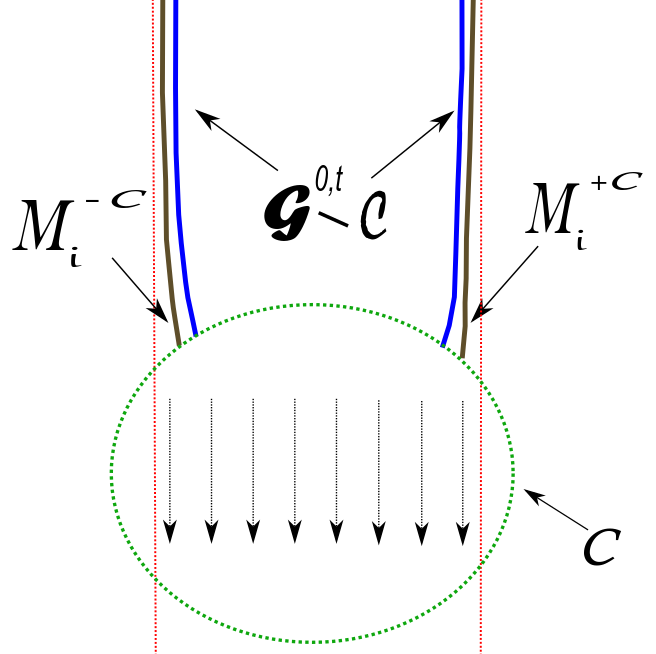}
\end{center}
\caption{Transversal section of the behaviour of $M$ and $\mathcal{G}^{0,t}$ in the left side and transversal section of the behaviour of $M\setminus\mathcal{C}$ and $\mathcal{G}^{0,t}\setminus \mathcal{C}$ in the right side.} \label{figure 9}
\end{figure}

Now take any cylinder $\mathcal{C}$ so that $M\setminus \mathcal{C}$ and $G^{0,t}\setminus \mathcal{C}$ have two connected components, which are $C^{1}-$asymptotic to two half-hyperplanes. Clearly the family 
\begin{equation*}
\left\{\mathcal{G}^{0,t}_{\mathcal{C}}:=\mathcal{G}^{0,t}\setminus \mathcal{C}-i \cdot u_{\theta}\right\}_{i\in \mathbb{N}}
\end{equation*}
converges which respect to the $C^{\infty}-$topology to two hyperplanes $\Pi\left(-\frac{\pi}{2\cos(\theta)}\right)$ and $\Pi\left(\frac{\pi}{2\cos(\theta)}\right)$. This forces the family 
$
\textstyle{\left\{M_{i}^{\mathcal{C}}:=M\setminus \mathcal{C}-i \cdot u_{\theta}\right\}_{i\in \mathbb{N}}}
$
to converge as sets to these hyperplanes (see Figure \ref{figure 9}). By Theorem \ref{Shahriyari-Xin} each component of $M_{i}^{\mathcal{C}}$ is stable. Let $M_{i}^{-\mathcal{C}}$ and $M_{i}^{+\mathcal{C}}$ be the two connected components of $M_{i}^{\mathcal{C}}$ that lie on the left side of the hyperplane $\Pi(0)$ and on the right side of the hyperplane $\Pi(0)$, respectively. It is important to observe that 
\[\displaystyle{\operatorname{Area}(M_{i}\cap K)\leq C(K)}\]for every compact set $K\subset \R^{n+1}$ where $C(K)$ is a constant only depending on $K$ and the dimension $n$, and the area is to take with respect Ilmanen's metric. Therefore, by Theorem \ref{Strong Compactness Theorem}, the sequences $\left\{M_{i}^{-\mathcal{C}}\right\}$ and $\left\{M_{i}^{+\mathcal{C}}\right\}$ converge in the $C^{\infty}-$topology (with multiplicity one) to the hyperplanes $\Pi\left(-\frac{\pi}{2\cos(\theta)}\right)$ and $\Pi\left(\frac{\pi}{2\cos(\theta)}\right)$, respectively,  away from the singular set, with respect to the Ilmanen's metric. Now by Theorem \ref{Allard's Regularity Theorem}, we obtain that the convergence is $C^{\infty}$ with multiplicity one everywhere. This proves that $M$ is $C^{\infty}$-asymptotic to the half-hyperplanes of $\Pi\left(-\frac{\pi}{2\cos(\theta)}\right)$ and $\Pi\left(\frac{\pi}{2\cos(\theta)}\right),$ with respect to the Ilmanen's metric. 

Finally, we claim that $M$ is $C^2-$asymptotic to the half-hyperplanes $\Pi\left(-\tfrac{\pi}{2\cos(\theta)}\right)$ and $\Pi\left(\frac{\pi}{2\cos(\theta)}\right).$ Let us work with the wing of $M$ which is $C^{1}-$close to the half-hyperplane $\mathcal{H}_{1}$ of $\Pi\left(\tfrac{\pi}{2\cos(\theta)}\right).$ As we know, given $\epsilon>0,$ there exists $\delta>0$ so that $M$ can be represent a graph of $\varphi$ defined over $\mathcal{H}_{1}$, with $\sup_{\mathcal{H}_{1}(\delta)}|\varphi|<\epsilon$ and $\sup_{\mathcal{H}_{1}(\delta)}|D\varphi|<\epsilon$. Arguing by contraction, from the definition of $C^2$-asymptotic implies there exist $\epsilon>0$ and a sequence $\left\{p_{i}\right\}$ in $M$ such that: 
\begin{equation}
|D^2\varphi(p_{i})|\geq\epsilon\ \operatorname{and}\ \langle p_{i},u_{\theta}\rangle\to\infty.
\end{equation}
Consider the sequence $\left\{M_{i}:=M-p_{i}\right\}.$ The argument above shows that the wings of $M_i$ which are asymptotic to $\mathcal{H}_{1}-p_i$ converge to $\Pi(0)$ in the $C^{\infty}-$topology, with respect to Ilmanen's metric Notice that this wing is graph of a function $\varphi_{i}(\cdot)=\varphi(\cdot+p_{i})-\varphi(p_{i})$ which is defined in a half-hyperplane of $\Pi(0)$ that contain the origin. 

Consider a small geodesic cylinder $W_{r,\epsilon}$ around $0 \in \R^{n+1}$, with respect to Ilmanen's metric. By definition of convergence in the $C^{\infty}-$topology, there exist sufficiently large $i_{0}\in \mathbb{N}$ so that for all $i>i_0$ the set $W_{r,\epsilon}\cap M_i$ is a graph of a function $\eta_{i}$ defined over $B_{r}(p)\subset\Pi(0)$ such that $\sup_{B_{r}\cap\Pi(0)}|D^l\eta_{i}|<\epsilon/8$, for all $l\in \mathbb{N}$. Notice that the hyperplanes parallel to $\ee_{n+1}$ are totally geodesic and $\ee_{1}$ is normal vector to $T_{0}\Pi(0)$ and we have the following relation between $\varphi_{i}$ and $\eta_i$:
\begin{equation*}
\varphi_{i}(\exp_{0}(q+\eta_{i}(q)\ee_{1})-\langle\exp_{0}(q+\eta_{i}(q)\ee_{1}),\ee_{1}\rangle \ee_{1})=\langle\exp_{0}(q+\eta_{i}(q)\ee_{1}),\ee_{1}\rangle.
\end{equation*}

Differentiating twice and evaluating at $q=0$, we deduce that
\begin{eqnarray*}
\langle D^2\exp_{0}(\overline{u}_i,\overline{w}_i),\ee_1\rangle+D^2\eta_{i}(u,w)&=& D^2\varphi_{i}(u,w)\\
\nonumber &+&\textstyle{d\varphi_{i}[D^2\exp_{0}(\overline{u}_i,\overline{w}_i)-\langle D^2\exp_{0}(\overline{u}_i,\overline{w}_i),\ee_{1}\rangle \ee_{1}],}  
\end{eqnarray*}
where $\overline{u}_i:=u+d\eta_{i}(u)\ee_{1}.$ From this expression, the control on the $C^\infty$ norm of $\eta_{i}$ and on the $C^1$ norm of $\varphi$ we get a contraction with $|D^2\varphi_i(0)|=|D^2\varphi(p_{i})|\geq\epsilon,$  if $i$ is sufficiently large. This proves the lemma.
\end{proof}

For the sake of simplicity, let us set $h_{j}:=\tfrac{\langle\xi,\ee_{j}\rangle}{H}$, where $j\in\{2,\ldots,n-1\}$ and $h_{n}=\tfrac{\langle\xi,E_{n}\rangle}{H}$ (recall that $E_{n}=\cos(\theta)\ee_{n}+\sin(\theta)\ee_{n+1}.$)
Using the previous lemma we can obtain some information about the behaviour of the functions $h_{j}$ at the ends of $M$. 

\begin{Lemma} \label{h}
The functions $h_{j}$, $j \in\{2,\ldots,n\},$ tend to zero as we approach the end of $M.$
\end{Lemma}
\begin{proof} Consider the exhaustion $\left\{\Lambda_{i}\right\}$ given by \eqref{lambdai}. Notice that the boundary of each $\Lambda_{i}$ consists of the following $2n-1$ regions
\begin{eqnarray*}
\Lambda_{i}^{1}:=\{x=(x_{1},\ldots,x_{n+1})\in M \;:(x_{2},\ldots,x_{n-1})\in \left[-a_{i},b_{i}\right]^{n-2}\nonumber\\
\langle x,E_{n}\rangle\in \left[-a_{i},b_{i}\right]\;,\langle x,u_{\theta}\rangle=i\}
\end{eqnarray*}
\begin{eqnarray*}
\Lambda_{i}^{-2}:=\{x=(x_{1},\ldots,x_{n+1})\in M :\ (x_{3},\ldots,x_{n-1})\in[-a_{i},b_{i}]^{n-3}\;,\langle x,u_{\theta}\rangle\leq i\nonumber\\
x_{2}=-a_{i}\;,\langle x,E_{n}\rangle\in[-a_{i},b_{i}]\}
\end{eqnarray*}
\begin{eqnarray*}
\Lambda_{i}^{+2}:=\{x=(x_{1},\ldots,x_{n+1})\in M :\ (x_{3},\ldots,x_{n-1})\in[-a_{i},b_{i}]^{n-3}\;,\langle x,u_{\theta}\rangle\leq i\nonumber\\
x_{2}=b_{i}\;,\langle x,E_{n}\rangle\in[-a_{i},b_{i}]\}
\end{eqnarray*}
\begin{equation*}
\vdots
\end{equation*}
\begin{eqnarray*}
\Lambda_{i}^{-n}:=\{x=(x_{1},\ldots,x_{n+1})\in M :\ (x_{2},\ldots,x_{n-1})\in[-a_{i},b_{i}]^{n-2}\nonumber\\
 \langle x,u_{\theta}\rangle\leq i\;,\langle x,E_{n}\rangle=-a_{i}\}
\end{eqnarray*}
and
\begin{eqnarray*}
\Lambda_{i}^{+n}:=\{x=(x_{1},\ldots,x_{n+1})\in M :\ (x_{2},\ldots,x_{n-1})\in[-a_{i},b_{i}]^{n-2}\nonumber\\
\langle x,u_{\theta}\rangle\leq i\;,\langle x,E_{n}\rangle=b_{i}\}
\end{eqnarray*}
Let us study the behaviour of $h_{j}$ in a small strip around the boundary of $\Lambda_{i}$. We begin our analysis in the connected component $\Lambda_{i}^{1}$. Fix any sufficiently small $\epsilon >0 $. Taking into account  Lemma \ref{lem:grimreaper} and the definition of $M^{+}(t)$, we can use a similar argument as in the proof of Lemma \ref{me} to guarantee the existence of a sufficiently large $i_{1}(>i_{0})$, a sufficiently small $\tau>0$ and a smooth function  $\varphi$ defined on the strip \[\displaystyle{\mathcal{S}_\tau:=\left[\left(-\tfrac{\pi}{2\cos(\theta)},-\tfrac{\pi}{2\cos(\theta)}+\tau\right)\cup\left(\tfrac{\pi}{2\cos(\theta)}-\tau,\tfrac{\pi}{2\cos(\theta)}\right)\right]\times\mathbb{R}^{n-1}}\] satisfying
\begin{equation} \label{eq:tt}
\sup_{\mathcal{S}_{\tau}}|\varphi|<\epsilon, \quad \sup_{\mathcal{S}_{\tau}}|D\varphi|<\epsilon\ \mbox{ and} \; \ \sup_{\mathcal{S}_{\tau}}|D^2\varphi|<\epsilon
\end{equation}
and such that $M^{+}(i_{1})$ is graph of this function over the corresponding strip in the tilted grim reaper cylinder. From (\ref{normal}) and (\ref{curvature}) we obtain 
\begin{equation}\label{equality}
h_{j}=\frac{\alpha(j)}{\cos(\theta)}\frac{\partial_{x_{j}}\varphi}{\cos(x_{1}\cos(\theta))}\frac{1+\varphi\cos(\theta)\cos(x_{1}\cos(\theta))}{1+\sigma(\varphi,D\varphi)}, 
\end{equation}
where $\alpha(j)=(-1)^{j\left[n-\tfrac{j+1}{2}\right]},$ if $j\in\{2,\ldots,n-1\}$ and $\alpha(n)=(-1)^{n-1}\cos(\theta)$. Here 
\begin{eqnarray*}
\sigma(\varphi,D\varphi)&:=&\varphi\cos(\theta)\cos(x_{1}\cos(\theta))\nonumber\\
 &+&(-1)^{n}\{\sin(\theta)(1+\varphi\cos(\theta))\partial_{x_{n}}\varphi+\sin(x_{1}\cos(\theta))\partial_{x_{1}}\varphi\}
\end{eqnarray*}
Using the fact that $M^{+}(i_{1})$ is a graph over the tilted grim reaper cylinder and it is $C^2-$asymptotic to the half-hyperplane, we conclude that for all fixed $(x_{2},\ldots,x_{n})$ we have
\begin{equation*}
\lim_{x_{1}\to\tfrac{\pi}{2\cos(\theta)}^{-}}|\varphi|=\lim_{x_{1}\to\tfrac{\pi}{2\cos(\theta)}^{-}}|D\varphi|=0.
\end{equation*}
Therefore
\begin{equation}\label{des1}
|\partial_{x_{j}}\varphi(x_{1},x_{2},\ldots,x_{n})|=\left|-\int_{x_{1}}^{\tfrac{\pi}{2\cos(\theta)}}\partial_{x_{j}x_{1}}\varphi(x,x_{2},\ldots,x_{n})dx\right|\leq\left(\tfrac{\pi}{2\cos(\theta)}-x_{1}\right)\epsilon.
\end{equation}
So, by (\ref{eq:tt}), (\ref{equality}) and (\ref{des1}) we obtain 
that $
|h_{j}(x)|<o(\epsilon),
$
for all $x$ near $\Lambda_{i}^{1}$. Thus 
\begin{equation}\label{estimate1}
\sup_{N\left(\Lambda_{i}^{1}\right)}|h_{j}|<o(\epsilon)	
\end{equation}
where $N\left(\Lambda_{i}^{1}\right)$ is a small neighborhood the $\Lambda_{i}^{1}$ in $\Lambda_{i}$, if $i>i_{1}$. \

Now we are going to work with the components of $\partial \Lambda_i$ that intersect $M^{-}(i_{1})$. Since $R_{i}$ and $L_{i}$ are $C^{1}$-close  to a strip in the tilted grim reaper cylinder, there is sufficiently large $i_{2}$ such that $\textstyle{R_{i}\cap\left\{(x_{1},\ldots,x_{n+1})\in \mathbb{R}^{n+1};\langle x,u_{\theta}\rangle\leq i_{1}\right\}}$ is a graph over the strip in the tilted grim reaper cylinder of a function $\varphi_{i}$ defined in the strip \[\displaystyle{G_{\tau}:=\left(-\frac{\pi}{\cos(\theta)2}+\frac{\tau}{2},\frac{\pi}{2\cos(\theta)}-\frac{\tau}{2}\right)\times (m_{i}^{1},m^{2}_{i})^{n-1}}\] satisfying the following properties
\begin{equation}\label{des2}
\sup_{G_{\tau}}|\varphi_{i}|<\epsilon\ \operatorname{and} \sup_{G_{\tau}}|D\varphi_{i}|<\epsilon.
\end{equation}
The same estimate is true for $L_{i}.$ Furthermore, since $\cos (x_{1}\cos(\theta))> \kappa>0$  in $G_{\tau}$, for a suitable constant $\kappa,$ then (\ref{des2}) and (\ref{equality}) gives us that
$
\textstyle{\sup_{G_{\tau}}|h_{j}|<o(\epsilon)}.
$
Hence
\begin{equation}\label{estimate2}
\sup_{N\left(\Lambda_{i}^{\pm k}\right)}|h_{j}|<o(\epsilon),
\end{equation}
where $k\in\{2,\ldots,n\}$ and $N\left(\Lambda_{i}^{\pm k}\right)$ is a small neighbourhood of the $\Lambda_{i}^{\pm k}$ in $\Lambda_{i}.$ Hence for (\ref{estimate1}) and (\ref{estimate2}) we have $\textstyle{\sup_{N\left(\partial\Lambda_{i}\right)}|h_{j}|<o(\epsilon)}$, for any $i \in \n$, $i>\max\{i_{1},i_{2}\}$.
\end{proof}

\begin{flushleft}
Now we are ready to prove the main theorem of this paper.
\end{flushleft}
\begin{proof}[Proof of Theorem \ref{th:41}] Recall that we are assuming that $M$ is asymptotic to half-hyperplanes that are contained in different hyperplanes and that $\inf_M(\langle x,u_{\theta}\rangle)=0$. According to Lemma \ref{h} there is an interior point where $h_{j}$ has an extremum. Then, because $h_j$ is a solution of \eqref{tercera}, we can apply  Hopf's maximum principle to conclude that $h_{j}=0,$ that is, $\xi_{j}=0$ on $M$ for all $j \in \left\{2,\ldots,n\right\}$. In particular, each $\ee_{j}$ and $E_{n}$ are tangent vectors of $M$ for $j \in \left\{2,\ldots,n-1\right\}$ at all point of $M$. Thus, we can consider a global orthonormal frame in $M$, $\left\{E_{1},\ E_{j}=\ee_{j};\ j \in \left\{2,\ldots,n-1,\right\};E_{n} \right\}$ (see \eqref{Part1}), where $E_{1}=E_{2}\wedge\ldots\wedge E_{n}\wedge \xi$. 
Differentiating each $\xi_{j}$, $j\in \left\{2,\ldots, n\right\}$, with respect to $E_{k}$, $k\in \left\{1,\ldots, n\right\}$, we deduce
\begin{equation*}
0=E_{k}(\xi_{j})=E_{k}\langle\xi,E_{j}\rangle=\langle D_{E_{k}}\xi,E_{j}\rangle=A(E_{k},E_{j}).
\end{equation*}
Thus
\begin{equation*}
|A|^{2}=\sum_{i,j}A(E_{j},E_{k})^{2}=A(E_{1},E_{1})^{2}=H^{2}.
\end{equation*}
Therefore, by \cite[Theorem B]{MSHS15} we conclude that $M=\mathcal{G}^{0,0}$, because $0=\textstyle{\inf_{M}\langle x,u_{\theta}\rangle}$.
\end{proof}

\subsection{The case $\theta=\pi/2$.}

Now we are going to work in the case when the cylinder is vertical, i.e. the axis of the cylinder is parallel to the translating velocity.  The philosophy now is to compare $M$ with a suitable translation of itself to arrive a contraction. Then we need that the limit of the translated hypersurfaces is also smooth, but this holds true if $n<7$. However, we conjecture that the result is true for any dimension. 
First of all, we would like to point out the following:

\begin{Remark}\label{remarkcompctness}
The Dynamics Lemma (Lemma \ref{Compactness Lemma}) is still true in this situation, for every $n\geq 2.$ The proof works exactly as in the case $\theta<\pi/2$, except the proof that the sequence $\{M_i^{-}(\delta)\}$ has locally bounded area. In this case, in order to prove that the  area blow-up set is empty we use as barriers the family $P_\lambda=W^2_\lambda \times \R^{n-2}$ (cylinders over the translating catenoid of dimension $2$), for a sufficiently large $\lambda>0$ so that the cylinder lies inside the neck of $P_\lambda=W^2_\lambda \times \R^{n-2}$. Hence, if the set of area blow-up is not empty, then we could move $P_\lambda=W^2_\lambda \times \R^{n-2}$ until we get a first contact point with the  area blow-up set, which is impossible. 
\end{Remark}

\begin{Theorem}[Vertical Cylinders]\label{limitcase}
Let $M^n\subset \mathbb{R}^{n+1}$ be a complete, connected, properly embedded translating soliton and $\displaystyle{\mathcal{C}:=\{x\in \mathbb{R}^{n+1} \; : \; \langle x,\ee_{1}\rangle^2+\langle x,  \ee_{n}\rangle^2 \leq r^2\}},$ for $r>0.$ Assume that $M$ is $C^{1}$-asymptotic to two half-hyperplanes outside $\mathcal{C}$ and $n<7$. Then $M$ must coincide with a hyperplane parallel to $\ee_{n+1}.$
\end{Theorem}
\begin{proof} We claim that $\mathcal{H}_{1}$ and $\mathcal{H}_{2}$ are parallel. Assume to the contrary that is true. Then we could take a hyperplane parallel to $\ee_{n+1}$, $\Gamma$, such that it does not intersect $M$ and such that the normal vector $v$ to $\Gamma$ is not perpendicular to $w_{1}$ and $w_{2}$. Translate $\Gamma$ by $t_0 \in \R$ in the direction of $v$ until we get a hyperplane $\Gamma_{t_0}:=\Gamma+t_{0}v$ such that either $\Gamma_{t_0}$ and $M$ have a first point of contact or $\dist \left(\Gamma_{t_0},M\right)=0$ and $\Gamma_{t_0}\cap M=\varnothing.$ The first case is not possible by Theorem \ref{Strong Maximum Principle}. Regarding the second case, we can reason as in Lemma \ref{Characterization of the Hyperplane} to see that this case is also impossible.

Notice that we can not have either $\mathcal{H}_{1}\subset\mathcal{H}_{2}$ or $\mathcal{H}_{2}\subset\mathcal{H}_{1}$, because in these cases we could take a hyperplane parallel to $\ee_{n+1}$, $\Upsilon$, whose normal is exactly $w_{1}$ and do not intersect $M.$ Now we could move $\Upsilon$ into direction of $w_{1}$ until there exists $t_{0}>0$ such that either $\Upsilon+t_{0}w_{1}$ and $M$ have a first point of contact or $\left\{\Upsilon+t_{0}w_{1}\right\}\cap M=\varnothing$ and $\dist\left(\Upsilon+t_{0}w_{1}, M\right)=0$. Reasoning as in the above paragraph, we can conclude that booth situations are impossible. 

\begin{Remark}
We would like to point out that until this point the argument is valid for any dimension. In the remaining part of the proof we shall need that $n<7$.
\end{Remark}

Denote by $\Pi_{1}$ and $\Pi_{2}$ the hyperplanes that contain the half-hyperplane $\mathcal{H}_{1}$ and $\mathcal{H}_{2},$ respectively. By the previous claims, $\Pi_{1}$ and $\Pi_{2}$ are parallel. We claim that  $\Pi_{1}$ and $\Pi_{2}$ coincide. Notice that this proves our theorem, because if we reason as at the end of Lemma \ref{Characterization of the Hyperplane} we can deduce that $M$ coincides $\Pi_{1}(=\Pi_{2}.)$ 

Suppose to the contrary that $\Pi_{1}\neq\Pi_{2}$. Let $\nu$ be the normal vector to $\Pi_{1}$ and take $s_{0}$ sufficiently large so that $M+s_{0}\nu$ does not intersect the slab limited by $\Pi_{1}$ and $\Pi_{2}.$ Now consider a sufficiently large $t_{0}>0$ so that $(M\cap \mathcal{C})+s_{0}\nu+t_{0}w_{1}$   lies in $\mathcal{Z}_{1,\delta}^{+}$ (see Lemma \ref{Compactness Lemma}) and notice that the wing of $M$ which corresponds to $\mathcal{H}_{1}(\delta)$ is a graph over this set. 
Define the set \[\displaystyle{\mathcal{A}:=\{s\in[0,\infty); \{M+s\nu+t_{0}w_{1}\}\cap M=\varnothing\}}.\] Let $s_{1}:=\inf\mathcal{A}>0.$ We have two possibilities: either $s_{1}\notin\mathcal{A}$ or $s_{1}\in\mathcal{A}$. The first case implies that $M+s_{1}\nu+t_{0}w_{1}$ and $M$ have a point of contact, which is impossible by Theorem \ref{Strong Maximum Principle}. In the second case, we must have \[\displaystyle{\dist\left(M+s_{1}\nu+t_{0}w_{1},M\right)}=0\] and $\textstyle{\{M+s_{1}\nu+t_{0}w_{1}\}\cap M=\varnothing}.$ This means that there exist sequences $\{p_{i}\}$ in $M\cap\mathcal{Z}_{1,\delta}^{+}$ and $\{q_{i}\}$ in $\{M+s_{1}\nu+t_{0}w_{1}\}\cap\mathcal{Z}_{1,\delta}^{+}$ such that $\dist\{p_{i},q_{i}\}=0,$ note that we can suppose $\{\langle q_{i},\ee_{1}\rangle\}, \{\langle p_{i},\ee_{1}\rangle\}\to a$ and $\{\langle q_{i},\ee_{n}\rangle\}, \{\langle p_{i},\ee_{n}\rangle\}\to b$. Consider the sequences \[\displaystyle{\{M_{i}:=M-(0,p_{2},\ldots,p_{n-1},0,p_{n+1})\}}\] and \[\displaystyle{\{\widehat{M}_{i}:=M+s_{1}\nu+t_{0}w_{1}-(0,q_{2},\ldots,q_{n-1},0,q_{n+1})\}}.\] By Remark \ref{remarkcompctness}, up to a subsequence, $M_{i}\rightharpoonup M_{\infty}$ and $\widehat{M}_{i}\rightharpoonup\widehat{M}_{\infty},$ where $M_{\infty}$ and $\widehat{M}_{\infty}$ are connected stationary integral varifold, and $(a,0,\ldots,0,b,0)\in\operatorname{spt} M_{\infty}\cap\operatorname{spt} \widehat{M}_{\infty}$. Furthermore, by Theorem \ref{Strong Compactness Theorem} and our hypothesis, the point $(a,0,\ldots,0,b,0)$ lies in the region where $M_{\infty}$ is smooth. Hence Theorem \ref{Strong Maximum Principle} implies that $M_{\infty}=\widehat{M}_{\infty},$ which is impossible. Therefore $\Pi_{1}=\Pi_{2}=M.$
\end{proof}

\bibliographystyle{amsplain, amsalpha}

\end{document}